%% file: vN_acc.tex
\documentclass[oneside,english,reqno]{amsart}
\renewcommand{\familydefault}{\sfdefault}
\usepackage[T1]{fontenc}
\usepackage[latin9]{inputenc}
\usepackage{color}
\usepackage{babel}
\usepackage{amsthm}
\usepackage{amssymb}
\usepackage{graphicx}
\usepackage[unicode=true,pdfusetitle,
 bookmarks=true,bookmarksnumbered=false,bookmarksopen=false,
 breaklinks=false,pdfborder={0 0 1},backref=false,colorlinks=false]
 {hyperref}
\usepackage{breakurl}

\makeatletter

\providecommand{\tabularnewline}{\\}

\numberwithin{equation}{section}
\numberwithin{figure}{section}
\newenvironment{lyxlist}[1]
{\begin{list}{}
{\settowidth{\labelwidth}{#1}
 \setlength{\leftmargin}{\labelwidth}
 \addtolength{\leftmargin}{\labelsep}
 }}
{\end{list}}
\theoremstyle{plain}
\newtheorem{thm}{\protect\theoremname}[section]
  \theoremstyle{plain}
  \newtheorem{algorithm}[thm]{\protect\algorithmname}
  \theoremstyle{remark}
  \newtheorem{rem}[thm]{\protect\remarkname}
  \theoremstyle{plain}
  \newtheorem{prop}[thm]{\protect\propositionname}
  \theoremstyle{plain}
  \newtheorem{cor}[thm]{\protect\corollaryname}
  \theoremstyle{definition}
  \newtheorem{defn}[thm]{\protect\definitionname}
  \theoremstyle{remark}
  \newtheorem*{acknowledgement*}{\protect\acknowledgementname}

\newcommand{\stab}{\quad}
\newcommand{\intr}{\mbox{\rm int}}
\newcommand{\aff}{\mbox{\rm aff}}
\newcommand{\conv}{\mbox{\rm conv}}

\newcommand{\formtab}{\qquad}

\makeatother

  \providecommand{\acknowledgementname}{Acknowledgement}
  \providecommand{\algorithmname}{Algorithm}
  \providecommand{\corollaryname}{Corollary}
  \providecommand{\definitionname}{Definition}
  \providecommand{\propositionname}{Proposition}
  \providecommand{\remarkname}{Remark}
\providecommand{\theoremname}{Theorem}

\begin{document}
\title[Improving distance steps in von Neumann algorithm]{Improving the distance reduction step in the von Neumann algorithm} 

\begin{abstract}
A known first order method to find a feasible solution to a conic
problem is an adapted von Neumann algorithm. We improve the distance
reduction step there by projecting onto the convex hull of previously
generated points using a primal active set quadratic programming (QP)
algorithm. The convergence theory is improved when the QPs are as
large as possible. For problems in $\mathbb{R}^{2}$, we analyze our
algorithm by epigraphs and the monotonicity of subdifferentials. Logically,
the larger the set to project onto, the better the performance per
iteration, and this is indeed seen in our numerical experiments.
\end{abstract}

\author{C.H. Jeffrey Pang}

\curraddr{Department of Mathematics\\ 
National University of Singapore\\ 
Block S17 08-11\\ 
10 Lower Kent Ridge Road\\ 
Singapore 119076 }

\email{matpchj@nus.edu.sg}

\date{\today{}}


\maketitle
\tableofcontents{}

\section{Introduction}

For $A\in\mathbb{R}^{m\times n}$, consider the linear inequality
system\begin{subequations}\label{eq:O_p_vN_pair} 
\begin{equation}
A^{T}y>0,\label{eq:perceptron-ineq}
\end{equation}
and its alternative 
\begin{equation}
Ax=0,\mathbf{1}^{T}x=1\mbox{ and }x\geq0,\label{eq:von-Neumann-ineq}
\end{equation}
\end{subequations}where $\mathbf{1}$ stands for the vector of all
ones. More generally, for a closed convex cone $K\subset\mathbb{R}^{n}$
with interior, consider the conic system \begin{subequations}\label{eq:G_p_vN_pair}
\begin{equation}
A^{T}y\in\intr(K^{*}),\label{eq:perceptron-ineq-1}
\end{equation}
and its alternative 
\begin{equation}
Ax=0,\bar{u}^{T}x=1\mbox{ and }x\in K.\label{eq:von-Neumann-ineq-1}
\end{equation}
\end{subequations}where $K^{*}:=\{z:z^{T}x\geq0\mbox{ for all }x\in K\}$
is the (positive) dual cone of $K$ and $\bar{u}$ is some point in
$\intr(K^{*})$. It is easy to see that \eqref{eq:O_p_vN_pair} is
the particular case of \eqref{eq:G_p_vN_pair} when $K=\mathbb{R}_{+}^{n}$
and $\bar{u}=\mathbf{1}$. An easy variant of Farkas's Lemma shows
that exactly one of \eqref{eq:perceptron-ineq} and \eqref{eq:von-Neumann-ineq}
is feasible. A similar result is easily seen to hold for \eqref{eq:G_p_vN_pair}.

The perceptron algorithm is a simple iterative algorithm that finds
a solution to system \eqref{eq:perceptron-ineq} if it is feasible.
The von Neumann algorithm is a simple iterative algorithm that finds
an approximate solution to \eqref{eq:von-Neumann-ineq} if it is feasible,
and it can give a solution $y$ to \eqref{eq:perceptron-ineq} if
\eqref{eq:von-Neumann-ineq} turns out to be infeasible.

We recall some of the history of the perceptron and von Neumann algorithms.
The perceptron algorithm was introduced in \cite{Rosenblatt58} for
solving classification problems in machine learning. The von Neumann
algorithm was privately communicated by von Neumann to Dantzig in
the late 1940s, and later studied by Dantzig \cite{Dantzig92,Dantzig92_conv}.
Block \cite{Block62} and Novikoff \cite{Novikoff62} showed that
when \eqref{eq:perceptron-ineq} is feasible, the perceptron algorithm
finds a solution to \eqref{eq:perceptron-ineq} after at most $1/[\rho(A)]^{2}$
iterations, where $\rho(A)$, a condition number defined in \cite{CheungCucker01_MP},
is defined by 
\[
\rho(A):=\left|\max_{\|y\|_{2}=1}\min_{j=1,\dots,n}\frac{a_{j}^{T}y}{\|a_{j}\|_{2}}\right|.
\]
When \eqref{eq:perceptron-ineq} is feasible, $\rho(A)$ is precisely
the width of the feasibility cone $\{y:A^{T}y\geq0\}$ as defined
in \cite{FreundVera99}. This condition number traces its roots to
\cite{Renegar_SIOPT95} (see also \cite{Renegar_MP95,PenaRenegar00}).
Epelman and Freund \cite{EpelmanFreund00} showed that the von Neumann
algorithm either computes an $\epsilon$-solution to \eqref{eq:von-Neumann-ineq}
in $O(\frac{1}{\rho(A)^{2}}\log(\frac{1}{\epsilon}))$ iterations
when \eqref{eq:von-Neumann-ineq} is feasible, or finds a solution
to the alternative system \eqref{eq:perceptron-ineq} in $O(1/\rho(A)^{2})$
iterations if \eqref{eq:perceptron-ineq} is feasible. They also treated
the generalized pair \eqref{eq:G_p_vN_pair}. (See also \cite{EpelmanFreund02}.)

Consider the problem
\begin{eqnarray}
 & \min_{x\in\mathbb{R}^{n}} & \frac{1}{2}\|Ax\|^{2}\label{eq:primal-QP}\\
 & \mbox{s.t.} & \bar{u}^{T}x=1\nonumber \\
 &  & x\in K.\nonumber 
\end{eqnarray}
It is clear that \eqref{eq:von-Neumann-ineq-1} is feasible if and
only if the objective value of \eqref{eq:primal-QP} is zero. Alternatives
for solving \eqref{eq:G_p_vN_pair} include the interior point algorithm
and the ellipsoid algorithm. Both the interior point and ellipsoid
methods are sophisticated algorithms that give a better complexity
bound, but require significant computational effort to perform each
iteration. 

For the case where $K=\mathbb{R}_{+}^{n}$ and $\bar{u}=\mathbf{1}$,
an active set quadratic programming (QP) algorithm can also be an
alternative. An active set QP algorithm can easily solve \eqref{eq:O_p_vN_pair}
if $m$ and $n$ are small, and the subproblems in each iteration
are easily solved when $m$ is small. For larger problems, an active
set QP algorithm is considered to be as efficient as the simplex method
in practice. An advantage of the active set QP method is that the
minimum of \eqref{eq:primal-QP} can be attained in finitely many
iterations. If the active set QP algorithm is used to find a feasible
solution to \eqref{eq:perceptron-ineq}, the algorithm can terminate
before the minimizer is found.

The perceptron and von Neumann algorithms are a first-order methods
for solving \eqref{eq:O_p_vN_pair} in that the computational effort
in each iteration is small, but one would need much more iterations
than a more sophisticated algorithm like the interior point method
or the ellipsoid method. For large scale problems, a first-order method
may be the only reasonable approach. Since the von Neumann algorithm
uses only matrix vector multiplications and do not solve linear systems,
it is also useful for sparse problems.

Problem \eqref{eq:von-Neumann-ineq} is a particular case of the problem
of finding whether the convex hulls of two sets of points overlap.
More precisely, for $A\in\mathbb{R}^{m\times n_{1}}$ and $B\in\mathbb{R}^{m\times n_{2}}$,
consider 
\begin{eqnarray}
 & \min_{x\in\mathbb{R}^{n_{1}},z\in\mathbb{R}^{n_{2}}} & \frac{1}{2}\|Ax-Bz\|^{2}\label{eq:von-Neumann-gen}\\
 & \mbox{s.t.} & \mathbf{1}^{T}x=1\nonumber \\
 &  & \mathbf{1}^{T}z=1\nonumber \\
 &  & x,z\geq0.\nonumber 
\end{eqnarray}
The sets are the vectors spanned by the columns of $A$ and $B$ respectively.
When $B$ is the zero vector of size $m\times1$, then \eqref{eq:von-Neumann-gen}
reduces to to \eqref{eq:primal-QP}. When the convex hull of these
two sets of points do not overlap, the classification problem is the
problem of finding a good separating hyperplane between these two
sets. Research in the classification problem has gone on to handle
the misclassifications of some of the points \cite{Scheinberg06}.

There are other accelerations of the perceptron and von Neumann algorithms
in the literature. A smoothed perceptron von Neumann algorithm was
studied in \cite{SoheiliPena12}, who were in turn motivated by the
smoothing techniques in \cite{Nesterov05_SIOPT}. A randomized rescaled
version of the perceptron algorithm was proposed in \cite{DunaganVempala06_MP}
that terminates in $O(m\log(\frac{1}{\rho(A)}))$ with high probability.
The randomized algorithm was extended to more general conic systems
in \cite{BelloniFreundVempala09_MOR}.

Another well-known algorithm for solving feasibility problems is the
method of alternating projections. The idea of using a QP as an intermediate
step to accelerate the method of alternating projections was studied
by the author for general feasibility problems in \cite{cut_Pang12,SHDQP,improved_SIP},
though the idea had been studied for particular cases in \cite{Pierra84,BausCombKruk06}
(intersection of an affine space and a halfspace), \cite{G-P98,G-P01}
(under smoothness conditions) and \cite{Fukushima82} (for the convex
inequality problem). For more information, we refer to the references
in the papers mentioned earlier, and highlight \cite{BB96_survey,EsRa11}
as well as \cite[Subsubsection 4.5.4]{BZ05} for more information
about the theory and method of alternating projections. It is natural
to ask whether a QP can accelerate algorithms for solving \eqref{eq:O_p_vN_pair}
and \eqref{eq:G_p_vN_pair}.

\subsection{Contributions of this paper}

We make use of the fact that it is relatively easy to project a point
onto the convex hull of a small number of points using an active set
QP algorithm to generalize the von Neumann algorithm. When the size
of the set that defines the convex hull to be projected on equals
two, then our algorithm becomes the setting of the von Neumann algorithm.
The size of this set can be chosen to be as large as one can reasonably
can to increase efficiency, as long as each iteration is still manageable.
(See lines 7-10 of Algorithm \ref{alg:enhanced-vN}, Remark \ref{rem:vN-particular-case}
and the subsequent discussion.) 

For the case of \eqref{eq:O_p_vN_pair}, the generalized algorithm
(Algorithm \ref{alg:enhanced-vN}) is a variant of an active set QP
algorithm, and it converges to a point in finitely many iterations
to find a $y$ satisfying \eqref{eq:perceptron-ineq} or an $x$ satisfying
\eqref{eq:von-Neumann-ineq}, whichever one is feasible. 

For the case of \eqref{eq:G_p_vN_pair}, Theorem \ref{thm:fin-conv-1}
proves that Algorithm \ref{alg:enhanced-vN} converges to a point
in finitely many iterations if $0\in\intr(S)$, where $S=\intr\{Ap:\bar{u}^{T}p=1,p\in K\}$,
and all previously identified points are kept. For the case when $m=2$
and $0$ is on the boundary of $S$, we show that the convergence
of $\{\|y_{i}\|\}_{i}$ in Algorithm \ref{alg:enhanced-vN} to zero
is linear with rate at worst $1/\sqrt{2}$ in Theorem \ref{thm:lin-conv-best},
and further analyze the behavior of Algorithm \ref{alg:enhanced-vN}
in Section \ref{sec:More-analysis} by appealing to epigraphs and
the monotonicity of subdifferentials.

\subsection{Notation}

We list down some common notation used in this paper, which are rather
standard material in convex analysis \cite{Rockafellar70}. Let $C\subset\mathbb{R}^{n}$
be a set.
\begin{lyxlist}{00.00.0000}
\item [{$\aff(C)$}] The affine hull of the set $C$.
\item [{$\conv(C)$}] The convex hull of the set $C$.
\end{lyxlist}
If $C$ is a closed convex set, we have the following notation.
\begin{lyxlist}{00.00.0000}
\item [{$T_{C}(x)$}] The tangent cone of the set $C$ at $x$.
\item [{$\partial C$}] The boundary of $C$.
\end{lyxlist}
For a convex function $f:\mathbb{R}\to\mathbb{R}$, we have the following
notation.
\begin{lyxlist}{00.00.0000}
\item [{$\partial f(x)$}] The subdifferential of $f$ at $x$, and $\partial f(x)$
is a subset of $\mathbb{R}$.
\end{lyxlist}

\section{Algorithm}

In this section, we propose Algorithm \ref{alg:enhanced-vN} for solving
\eqref{eq:G_p_vN_pair}, and describe the algorithmic issues incrementally.
 We start by describing Algorithm \ref{alg:enhanced-vN}.
\begin{algorithm}
\label{alg:enhanced-vN}(Algorithm for system \eqref{eq:G_p_vN_pair})
For $A\in\mathbb{R}^{m\times n}$ and a closed convex cone $K\subset\mathbb{R}^{n}$
with interior, this algorithm finds either a $y\in\mathbb{R}^{m}$
satisfying \eqref{eq:perceptron-ineq-1} or an $x\in\mathbb{R}^{n}$
satisfying \eqref{eq:von-Neumann-ineq-1}.

01 $\quad$Set $x_{0}=p_{0}\in\{x:\bar{u}^{T}x=1,x\in K\}$, $y_{0}=Ax_{0}$
and $i=0$

02 $\quad$Loop

03 $\quad$$\formtab$Find some $p_{i+1}$ such that $p_{i+1}\in K$,
$\bar{u}^{T}p_{i+1}=1$ and $p_{i+1}^{T}Ay_{i}\leq0$. 

04 $\quad$$\formtab$$\formtab$If $p^{T}A^{T}y_{i}>0$ for all $p\in K$
such that $\bar{u}^{T}p=1$, 

05$\quad$$\formtab$$\formtab$$\formtab$then $A^{T}y_{i}\in\intr(K^{*})$,
solving \eqref{eq:perceptron-ineq-1}, so we exit.

06 $\quad$$\formtab$\textbf{(Distance reduction)}

07 $\quad$$\formtab$Let $C_{i+1}=\{Ac_{i+1,1},\dots,Ac_{i+1,k_{i+1}}\}$
be a finite subset of 

08$\quad$$\formtab$$\formtab$$\conv\{Ap_{0},Ap_{1},\dots,Ap_{i+1}\}$.

09 $\quad$$\formtab$Let $y_{i+1}:=P_{\scriptsize\conv(\tilde{C})}(0)=\sum_{j=1}^{k_{i+1}}\lambda_{j}^{(i+1)}Ac_{i+1,j}$,
where $\tilde{C}\subset C_{i+1}$, 

10$\quad$$\formtab$$\formtab$$\sum_{j=1}^{k_{i+1}}\lambda_{j}^{(i+1)}=1$
and $\lambda_{j}^{(i+1)}\geq0$ for all $j\in\{1,\dots,k_{i+1}\}$.

11 $\quad$$\formtab$Let $x_{i+1}=\sum_{j=1}^{k_{i+1}}\lambda_{j}^{(i+1)}c_{i+1,j}$,
and $i\leftarrow i+1$.

12 $\quad$$\formtab$Perform an \textbf{aggregation step }to reduce
the size of $C_{i}$.

13 $\quad$until $\|y_{i}\|=\|Ax_{i}\|$ small.\end{algorithm}
\begin{rem}
(Choice of $p_{i+1}$) The $p_{i+1}$ in line 3 is typically chosen
by solving the conic section optimization problem 
\begin{eqnarray}
p_{i+1}= & \arg\min_{p} & p^{T}A^{T}y_{i}\label{eq:new-p-i-1-formula}\\
 & \mbox{s.t.} & \bar{u}^{T}p=1\nonumber \\
 &  & p\in K.\nonumber 
\end{eqnarray}
In the case where $\bar{u}=\mathbf{1}$ and $K=\mathbb{R}_{+}^{n}$,
the vector $p_{i+1}$ is easily seen to be the elementary vector $e_{j}$,
where $j$ corresponds to the coordinate of $A^{T}y_{i}$ with the
minimum value. In the case where $K$ is the semidefinite cone $\mathcal{S}_{+}^{k\times k}$
and $\bar{u}$ is the identity matrix in $\mathbb{R}^{k\times k}$,
the optimization objective is now $\langle p,A^{T}y_{i}\rangle$,
where $A^{T}$ is the adjoint of the operator $A:\mathcal{S}_{+}^{k\times k}\to\mathbb{R}^{m}$
and $\langle,\rangle$ corresponds to the trace inner product. A minimizer
of \eqref{eq:new-p-i-1-formula} is easily obtained once the eigenvalue
factorization of $A^{T}y_{i}$ is obtained. If $K$ is the direct
sum of sets of the form $\mathbb{R}_{+}^{n}$ and semidefinite cones,
the problem \eqref{eq:new-p-i-1-formula} can still be easily solved.
More elaboration is given in \cite{EpelmanFreund00} for example.
\end{rem}
We show that Algorithm \ref{alg:enhanced-vN} is an enhancement to
the von Neumann algorithm.
\begin{rem}
\label{rem:vN-particular-case}(von Neumann algorithm) The generalized
von Neumann algorithm to solve \eqref{eq:G_p_vN_pair} in \cite{EpelmanFreund00}
is a particular case of Algorithm \ref{alg:enhanced-vN} when the
$p_{i+1}$ in line 3 is chosen by \eqref{eq:new-p-i-1-formula} and,
most importantly, the set $C_{i+1}$ in line 7 is chosen to be $\{y_{i},Ap_{i+1}\}$.
Choices of $p_{i+1}$ different from \eqref{eq:new-p-i-1-formula}
were explored. See for example \cite{SoheiliPena13}. The classical
von Neumann algorithm is the particular case when $K=\mathbb{R}_{+}^{n}$,
\textbf{$\bar{u}=\mathbf{1}$}, $p_{0}=\frac{1}{n}\mathbf{1}$. When
$K$ is a product of semidefinite cones and cones of the type $\mathbb{R}_{+}^{k}$,
the optimization problem \eqref{eq:new-p-i-1-formula} is easy to
solve. (See for example \cite{EpelmanFreund00}.)
\end{rem}
The key generalization over the von Neumann algorithm in Algorithm
\ref{alg:enhanced-vN} is lines 7 to 10. The set $C_{i}$ is typically
taken to be of size 2, but we notice that it is still easy to project
onto the convex hull of a small number of points using an active set
quadratic programming (QP) algorithm. An active set QP algorithm is
considered to be as efficient as the simplex method in practice, and
we will describe the active set QP algorithm in Algorithm \ref{alg:active-QP}
later. 
\begin{rem}
(The $\tilde{C}$ in line 9) One would ideally choose $\tilde{C}=C_{i+1}$
in line 9 of Algorithm \ref{alg:enhanced-vN}, but it may take prohibitively
many iterations in order for the active set QP algorithm to find $P_{\scriptsize\conv(C_{i+1})}(0)$.
The active set QP algorithm finds $P_{\scriptsize\conv(\tilde{C})}(0)$
for different active sets $\tilde{C}$, ensuring a reduction in $d(0,\conv(\tilde{C}))$
at each iteration. (See Proposition \ref{prop:Facts-active-QP}, in
particular (4) and (5).) If the active set QP algorithm is expected
to take too many iterations before completion, then one can stop earlier
and find a new element to add to $C_{i}$ in line 7 of Algorithm \ref{alg:enhanced-vN}
instead.
\end{rem}
A useful property is the following.
\begin{rem}
\label{rem:finiteness-certificate}(Feasibility certificate) Suppose
\[
|\{\lambda_{j}^{(i)}:\lambda_{j}^{(i)}>0\}|=m+1,
\]
(i.e., there are $m+1$ positive terms in $\{\lambda_{j}^{(i)}\}_{j=1}^{k_{i}}$)
at line 12 of Algorithm \ref{alg:enhanced-vN}. The set $\{Ac_{i,j}:\lambda_{j}^{(i)}>0\}$
then contains $m+1$ points. If in addition, the affine space $\aff\{Ac_{i,j}:\lambda_{j}^{(i)}>0\}$
equals $\mathbb{R}^{m}$, then $0$ lies in $\intr\,\conv\{Ac_{i,j}:\lambda_{j}^{(i)}>0\}$.
This condition can be used as an effective certificate of the feasibility
of \eqref{eq:von-Neumann-ineq-1}. Similar ideas, referred to as bracketing,
were proposed in \cite{Dantzig92}.
\end{rem}

\subsection{\label{sub:agg_strat}Aggregation strategies}

In line 12 of Algorithm \ref{alg:enhanced-vN}, we provided an option
of performing an aggregation step to reduce the size of the set $C_{i}$
so that each iteration can be performed in a reasonable amount of
effort. We now show how this aggregation can be done. Recall that
$C_{i}=\{Ac_{i,1},\dots,Ac_{i,k_{i}}\}$. The iterate $y_{i}$ can
be written as 
\begin{equation}
y_{i}=\sum_{j=1}^{k_{i}}\lambda_{j}^{(i)}Ac_{i,j},\mbox{ where }\sum_{j=1}^{k_{i}}\lambda_{j}^{(i)}=1\mbox{ and }\lambda_{j}^{(i)}\geq0\mbox{ for all }j\in\{1,\dots,k_{i}\}.\label{eq:y-lambda-c}
\end{equation}
In other words, $y_{i}$ is a convex combination of some of the elements
in $C_{i}$. A logical first step to reduce the size of $C_{i}$ is
to discard indices $j$ such that $\lambda_{j}^{(i)}=0$. By Caratheodory's
theorem, the size of the set $C_{i}\subset\mathbb{R}^{m}$ can be
reduced to be at most $m+1$. If $|C_{i}|=m+1$, then it means that
$0$ lies in $\intr\,\conv(C_{i})$, which ends our algorithm (see
Remark \ref{rem:finiteness-certificate}).  If $|C_{i}|$ is still
not small enough so that the iterations of Algorithm \ref{alg:enhanced-vN}
can be easily performed, then we can aggregate to reduce the size
of $C_{i}$ by 1, as described below.
\begin{rem}
\label{rem:agg_strat}(Aggregation procedure 1) If the set $C_{i}=\{Ac_{i,1},\dots,Ac_{i,k_{i}}\}$,
the point $y_{i}$ and the vector $\lambda^{(i)}$ satisfy \eqref{eq:y-lambda-c}
and that $\lambda_{j}^{(i)}>0$ for all $j\in\{1,\dots,k_{i}\}$,
we can reduce the size of the active set $C_{i}$ by one using the
following procedure.
\begin{itemize}
\item Find 2 elements of $C_{i}$, say $\tilde{c}_{k_{i}-1}$ and $\tilde{c}_{k_{i}}$,
using one of the following strategies 

\begin{itemize}
\item The elements $\tilde{c}_{k_{i}-1}$ and $\tilde{c}_{k_{i}}$ are the
oldest elements not to have been aggregated.
\item The coefficients $\lambda_{k_{i}-1}^{(i)}$ and $\lambda_{k_{i}}^{(i)}$
are the smallest.
\item The coefficients $\lambda_{k_{i}-1}^{(i)}$ and $\lambda_{k_{i}}^{(i)}$
are the largest.
\end{itemize}
\item Set $\tilde{c}_{k_{i}-1}\leftarrow\frac{\lambda_{k_{i}-1}^{(i)}}{\lambda_{k_{i}-1}^{(i)}+\lambda_{k_{i}}^{(i)}}\tilde{c}_{k_{i}-1}+\frac{\lambda_{k_{i}}^{(i)}}{\lambda_{k_{i}-1}^{(i)}+\lambda_{k_{i}}^{(i)}}\tilde{c}_{k_{i}}$,
$\lambda_{k_{i-1}}^{(i)}\leftarrow\lambda_{k_{i-1}}^{(i)}+\lambda_{k_{i}}^{(i)}$,
$\lambda_{k_{i}}^{(i)}\leftarrow0$, and $C_{i}\leftarrow C_{i}\backslash\{\tilde{c}_{k_{i}}\}$.
\end{itemize}
\end{rem}
In order to reduce the set $C_{i}$ to a manageable size, one can
perform as many iterations of the procedure in Remark \ref{rem:agg_strat}
as needed to drop more points, or to amend the procedure to drop more
points per iteration.

Consider points $p_{i+1}$ obtained by the optimization procedure
\eqref{eq:new-p-i-1-formula}. The point $Ap_{i+1}$ minimizes $\{y_{i}^{T}v:v=Ap,\bar{u}^{T}p=1,p\in K\}$.
Hence $Ap_{i+1}$ lies on the boundary of the set $S:=\{Ap:\bar{u}^{T}p=1,p\in K\}$.
If we aggregate by taking some weighted average of some of the points
in $\{Ap_{0},\dots,Ap_{i+1}\}$, the new points obtained can lie in
the relative interior of $S$. It may be desirable to keep as many
points on the boundary of $S$ as possible, and we describe a second
aggregation procedure. 
\begin{rem}
\label{rem:agg_strat_2}(Aggregation procedure 2) For the setting
in Remark \ref{rem:agg_strat}, we can consider an alternative aggregation
strategy. 
\begin{itemize}
\item If no vector had been obtained by an aggregation process (i.e., if
there were no vectors of the form $\tilde{c}_{k_{i}-1}$ produced
at the end of Remark \ref{rem:agg_strat}),

\begin{itemize}
\item Perform the procedure in Remark \ref{rem:agg_strat}, but store the
aggregated vector (i.e., the vector $\tilde{c}_{k_{i}-1}$) in $\tilde{c}_{1}$
instead after some change of indices.
\end{itemize}
\item else

\begin{itemize}
\item Choose a vector, say $\tilde{c}_{k_{i}}$, by some criterion (for
example, the ones similar to Remark \ref{rem:agg_strat}), and aggregate
with the steps $\tilde{c}_{1}\leftarrow\frac{\lambda_{1}^{(i)}}{\lambda_{1}^{(i)}+\lambda_{k_{i}}^{(i)}}\tilde{c}_{1}+\frac{\lambda_{k_{i}}^{(i)}}{\lambda_{k_{i}-1}^{(i)}+\lambda_{k_{i}}^{(i)}}\tilde{c}_{k_{i}}$,
$\lambda_{1}^{(i)}\leftarrow\lambda_{1}^{(i)}+\lambda_{k_{i}}^{(i)}$,
$\lambda_{k_{i}}^{(i)}\leftarrow0$, and $C_{i}\leftarrow C_{i}\backslash\{\tilde{c}_{k_{i}}\}$.
\end{itemize}
\end{itemize}
\end{rem}

\subsection{Primal active set quadratic programming}

We now discuss the primal active set quadratic programming algorithm
for performing the projection $y_{i+1}=P_{\scriptsize\conv(C_{i+1})}(0)$
in line 9 of Algorithm \ref{alg:enhanced-vN}.
\begin{algorithm}
\label{alg:active-QP}(Active set QP algorithm for $y=P_{\scriptsize\conv(C)}(0)$)
Let $C=\{c_{1},\dots,c_{k}\}$ be a finite set, as in the setting
of line 9 in Algorithm \ref{alg:enhanced-vN}. We give the full details
of how to evaluate $y=P_{\scriptsize\conv(C)}(0)$, as well to find
the multipliers $\lambda\in\mathbb{R}^{k}$ such that $y=\sum_{j=1}^{k}\lambda_{j}c_{j}$,
$\sum_{j=1}^{k}\lambda_{j}=1$. 

01$\stab$Find $j_{0}^{*}\in\{1,\dots,k\}$.

02$\stab$Set $\lambda^{(0)}=e_{j_{0}^{*}}$, $J_{0}=\{j_{0}^{*}\}$,
$\tilde{C}_{0}=\{c_{j_{0}^{*}}\}$ and $\tilde{y}_{0}=c_{j_{0}^{*}}$. 

03$\stab$Set $i=1$.

04$\stab$Begin outer loop \textbf{{[}Find entering index{]}}

05$\stab$$\stab$Find an index $j_{i}^{*}$ such that $\tilde{y}_{i-1}^{T}c_{j_{i}^{*}}<\tilde{y}_{i-1}^{T}c$
for all (or for any) $c\in\tilde{C}_{i-1}$.

06$\stab$$\stab$$\stab$If such $j_{i}^{*}$ does not exist, then
$\tilde{y}_{i-1}=P_{\scriptsize\conv(C)}(0)$, and we end.

07$\stab$$\stab$Set $\tilde{C}=\tilde{C}_{i-1}$, $\tilde{J}=J_{i-1}$,
$\tilde{y}=\tilde{y}_{i-1}$, $\tilde{\lambda}=\lambda^{(i-1)}$.

08$\stab$$\stab$Begin inner loop \textbf{{[}Tries to add $j_{i}^{*}$
to active set{]}}

09$\stab$$\stab$$\stab$Find $y^{\prime}=P_{\scriptsize\aff(\tilde{C}\cup\{c_{j_{i}^{*}}\})}(0)$,
and write $y^{\prime}=[\sum_{j\in\tilde{J}}\lambda_{j}^{\prime}c_{j}]+\lambda_{j_{i}^{*}}^{\prime}c_{j_{i}^{*}}$ 

10$\stab$$\stab$$\stab$$\stab$for some $\lambda^{\prime}\in\mathbb{R}^{k}$
such that $\mathbf{1}^{T}\lambda^{\prime}=1$ and $\lambda_{j}^{\prime}=0$
if $j\notin[\tilde{J}\cup\{j_{i}^{*}\}]$.

11$\stab$$\stab$$\stab$Find the largest $\bar{t}\in(0,1]$ such
that $\tilde{\lambda}+\bar{t}[\lambda^{\prime}-\tilde{\lambda}]\geq0$

12$\stab$$\stab$$\stab$If $\bar{t}=1$, then \textbf{{[}Found better
active set{]}}

13$\stab$$\stab$$\stab$$\stab$Set $\tilde{J}\leftarrow\tilde{J}\cup\{j_{i}^{*}\}$,
$J_{i}\leftarrow\{j\in\tilde{J}:\lambda_{j}^{(i)}>0\}$, $\tilde{C}_{i}\leftarrow\{c_{j}\}_{j\in J_{i}}$. 

14$\stab$$\stab$$\stab$$\stab$Set $\lambda^{(i)}=\lambda^{\prime}$,
$\tilde{y}_{i}=y^{\prime}$ and exit inner loop

15$\stab$$\stab$$\stab$else\textbf{ {[}Remove non-active vertex{]}}

16$\stab$$\stab$$\stab$$\stab$At least one component of $\tilde{\lambda}+\bar{t}[\lambda^{\prime}-\tilde{\lambda}]$
equals zero, say $\tilde{j}$.

17$\stab$$\stab$$\stab$$\stab$Set $\tilde{y}\leftarrow\tilde{y}+\bar{t}[y^{\prime}-\tilde{y}]$,
$\tilde{\lambda}\leftarrow\tilde{\lambda}+\bar{t}[\lambda^{\prime}-\tilde{\lambda}]$
and $\tilde{J}\leftarrow\tilde{J}\backslash\{\tilde{j}\}$.

18$\stab$$\stab$$\stab$end

19$\stab$$\stab$$i\leftarrow i+1$

20$\stab$$\stab$end inner loop

21$\stab$end outer loop
\end{algorithm}
Some explanation of Algorithm \ref{alg:active-QP} is in order. The
following proposition collects some facts, some of which are analogues
of well known facts of active set quadratic programming algorithms.
\begin{prop}
\label{prop:Facts-active-QP}(Facts of Algorithm \ref{alg:active-QP})
For Algorithm \ref{alg:active-QP}, we have the following facts. At
each iteration $i$, 
\begin{enumerate}
\item $\lambda^{(i)}\in\mathbb{R}_{+}^{n}$, $\mathbf{1}^{T}\lambda^{(i)}=1$,
and $\tilde{y}_{i}=\sum_{j=1}^{k}\lambda_{j}^{(i)}c_{j}$.
\item $\lambda_{j}^{(i)}>0$ if and only if $j\in J_{i}$.
\item $\tilde{y}_{i}=P_{\scriptsize\conv(\tilde{C}_{i})}(0)$ and $\tilde{y}_{i}=\sum_{j=1}^{k}\lambda_{j}^{(i)}c_{j}=\sum_{j\in\tilde{J}_{i}}\lambda_{j}^{(i)}c_{j}$.
\item $\{j_{i+1}^{*}\}\subset J_{i+1}\subset J_{i}\cup\{j_{i+1}^{*}\}$.
\item {[}Improvement{]} $\|\tilde{y}_{i+1}\|<\|\tilde{y}_{i}\|$ unless
$\tilde{y}_{i}=P_{\scriptsize\conv(C)}(0)$. \\
(Note that $\|\tilde{y}_{i}\|=d(0,\conv(\tilde{C}_{i}))$.)
\item {[}Finite convergence{]} There is some $i^{*}$ such that $\tilde{y}_{i^{*}}=P_{\scriptsize\conv(C)}(0)$. 
\end{enumerate}
\end{prop}
For the formula $\tilde{y}_{i-1}^{T}c_{j_{i}}<\tilde{y}_{i-1}^{T}c$
for all $c\in\tilde{C}_{i-1}$ in line 5 of Algorithm \ref{alg:active-QP},
we note that since $\tilde{y}_{i-1}=P_{\tilde{C}_{i-1}}(0)$, 
\begin{equation}
\tilde{y}_{i-1}^{T}c_{j_{1}}=\tilde{y}_{i-1}^{T}c_{j_{2}}\mbox{ for all }j_{1},j_{2}\in J_{i-1}.\label{eq:y-c-any-index}
\end{equation}
Thus the ``for all'' there is equivalent to ``for any''. 

Algorithm \ref{alg:active-QP} works in the following manner. Property
(1) is clear. In view of property (2), $J_{i}$ is also the active
set. At each iteration, the pair $(\tilde{y}_{i},J_{i})$ satisfies
properties (2) and (3). Take any $\tilde{c}\in\tilde{C}_{i}$. If
$\tilde{y}_{i}^{T}c\geq\tilde{y}_{i}^{T}\tilde{c}$ for all $c\in C$,
then in view of \eqref{eq:y-c-any-index}, we can deduce (with a bit
of effort) that $\tilde{y}_{i}=P_{\scriptsize\conv(C)}(0)$. Otherwise,
we can find an index $j_{i+1}^{*}$ in line 5. For the next pair $(\tilde{y}_{i+1},J_{i+1})$,
we find $\tilde{y}_{i+1}=P_{\scriptsize\conv(\tilde{C}_{i}\cup\{c_{j_{i+1}^{*}}\})}(0)$,
but $\tilde{C}_{i}\cup\{c_{j_{i+1}^{*}}\}$ is not necessarily the
active set satisfying property (2). The removal of elements not in
the active set (through checking the sign of $\lambda$ in the inner
loop) gives us $J_{i+1}$ satisfying properties (2) and (4). Since
\[
\|\tilde{y}_{i+1}\|=d\big(0,\conv(\tilde{C}_{i}\cup\{c_{j_{i+1}^{*}}\})\big)<d\big(0,\conv(\tilde{C}_{i})\big)=\|\tilde{y}_{i}\|,
\]
property (5) is satisfied. Property (6) follows from property (5)
and the fact that the active set $J_{i}$ can take on only finitely
many possibilities.
\begin{rem}
(von Neumann algorithm via Algorithm \ref{alg:active-QP}) We show
how we can use Algorithm \ref{alg:active-QP} directly to solve the
system \eqref{eq:O_p_vN_pair} through \eqref{eq:primal-QP}. Let
$C=\{a_{1},\dots,a_{n}\}$ be the columns of $A$. It is clear that
$0$ lies in the convex hull of $C$ if and only if \eqref{eq:von-Neumann-ineq}
has a solution. The difference between using Algorithm \ref{alg:enhanced-vN}
for \eqref{eq:O_p_vN_pair} and using Algorithm \ref{alg:active-QP}
directly are 
\begin{enumerate}
\item Algorithm \ref{alg:enhanced-vN} can stop when it finds a $y$ such
that $A^{T}y>0$, whereas Algorithm \ref{alg:active-QP} as stated
would stop only at $P_{\scriptsize\conv(C)}(0)$.
\item Algorithm \ref{alg:enhanced-vN} allows for aggregation to reduce
the size of $C$ but not Algorithm \ref{alg:active-QP}.
\end{enumerate}
If no aggregation is performed in Algorithm \ref{alg:enhanced-vN},
then finite convergence follows from the fact that the active set
can take on finitely many possibilities and (5) of Proposition \ref{prop:Facts-active-QP}.
\end{rem}
We remark on our choice of the QP algorithm.
\begin{rem}
(Choice of QP algorithm) A QP can be solved by an interior point method
or by a dual active set QP algorithm \cite{Goldfarb_Idnani}. We believe
that our choice of a QP algorithm is most appropriate because of the
following consequence of Proposition \ref{prop:Facts-active-QP}(5):
If we expect that we still need many iterations to solve the QP, we
can abort the QP solver halfway and the iterate $\tilde{y}_{i}$ obtained
so far would be closer to $0$ than what we started with. We might
not be able to get such an improved iterate if other QP solvers were
used.
\end{rem}
We remark on how one can speed up the implementation of Algorithm
\ref{alg:active-QP}.
\begin{rem}
\label{rem:accelerate-primal-QP}(On projecting onto affine spaces
in Algorithm \ref{alg:active-QP}) Recall that in line 9 of Algorithm
\ref{alg:active-QP}, we need an algorithm to find $y^{\prime}=P_{\scriptsize\aff(C)}(0)$,
where $C$ is a set of points $\{c_{1},\dots,c_{k}\}$. Finding this
projection is equivalent to finding $\gamma\in\mathbb{R}^{k-1}$ such
that 
\[
y^{\prime}=c_{k}+\sum_{j=1}^{k-1}\gamma_{j}[c_{j}-c_{k}]\mbox{ and }y^{\prime}\perp[c_{j}-c_{k}]\mbox{ for all }j\in\{1,\dots,k-1\}.
\]
Let $\tilde{A}\in\mathbb{R}^{m\times(k-1)}$ be such that the $j$th
column is $[c_{j}-c_{k}]$ have the QR factorization $A=QR$. Then
one can easily figure that $\gamma=-R^{-1}Q^{T}c_{k}$. The bottleneck
in implementing Algorithm \ref{alg:active-QP} is thus to calculate
the QR factorization of matrices of the form $\tilde{A}$. One need
not calculate these QR factorizations from scratch, and can update
these QR factorizations whenever new columns are added or removed
using Given's rotations or Householder reflections. We refer the reader
to \cite{NW06} and the references therein for more details.
\end{rem}
More intuition is given in Figure \ref{fig:sample-runs-QP-alg},
where we show a sample run of Algorithm \ref{alg:active-QP} (or Algorithm
\ref{alg:enhanced-vN}) and the von Neumann algorithm. For this example
where $A\in\mathbb{R}^{2\times3}$, the von Neumann Algorithm takes
many iterations before it can certify the infeasibility of the QP,
while Algorithm \ref{alg:active-QP} finds the projection of $0$
onto $\conv\{a_{1},a_{2},a_{3}\}$ in two steps.

\begin{figure}[!h]
\begin{tabular}{|c|c|c|}
\hline 
\multicolumn{3}{|l|}{von Neumann Algorithm}\tabularnewline
\hline 
\hline 
\includegraphics[scale=0.2]{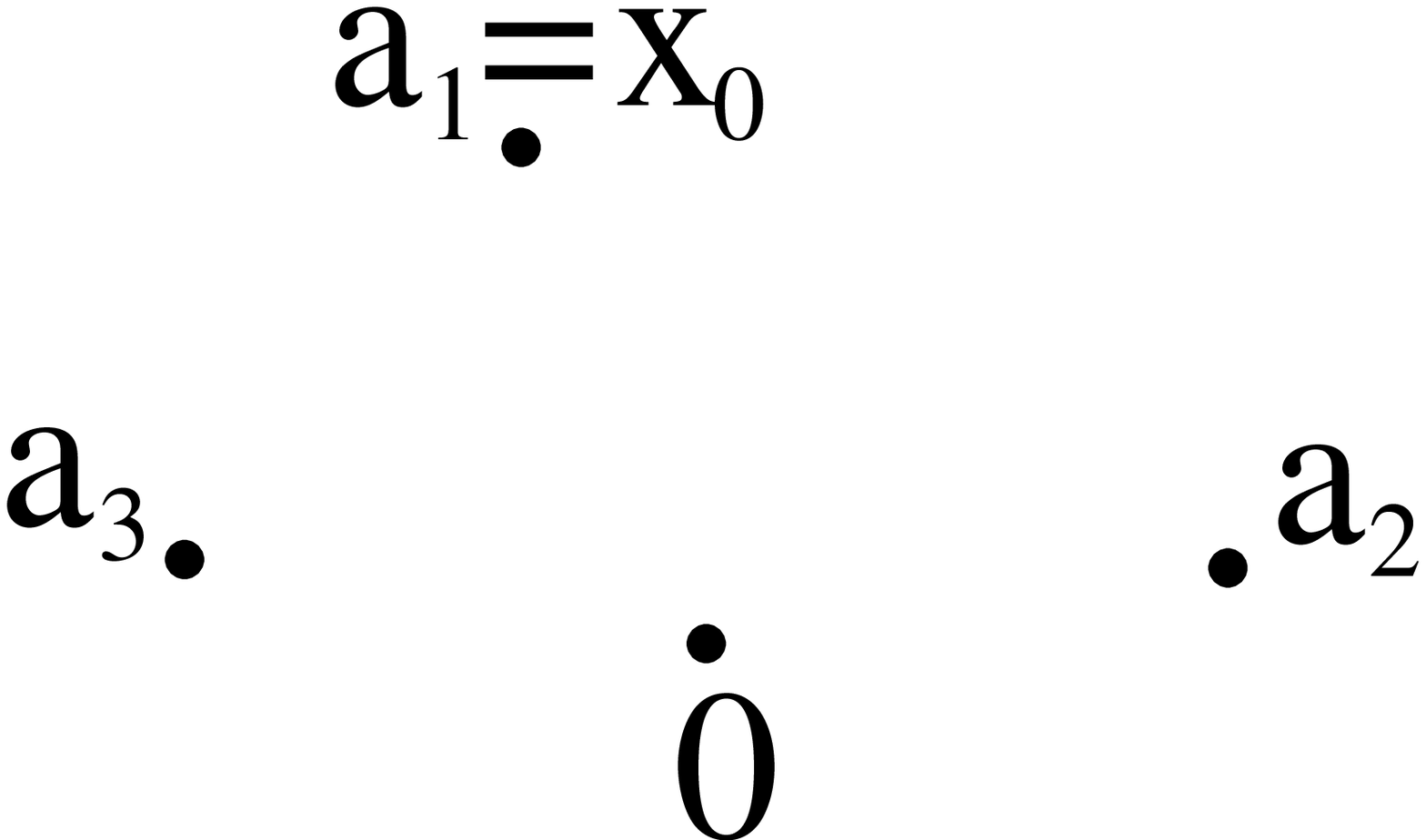} & \includegraphics[scale=0.2]{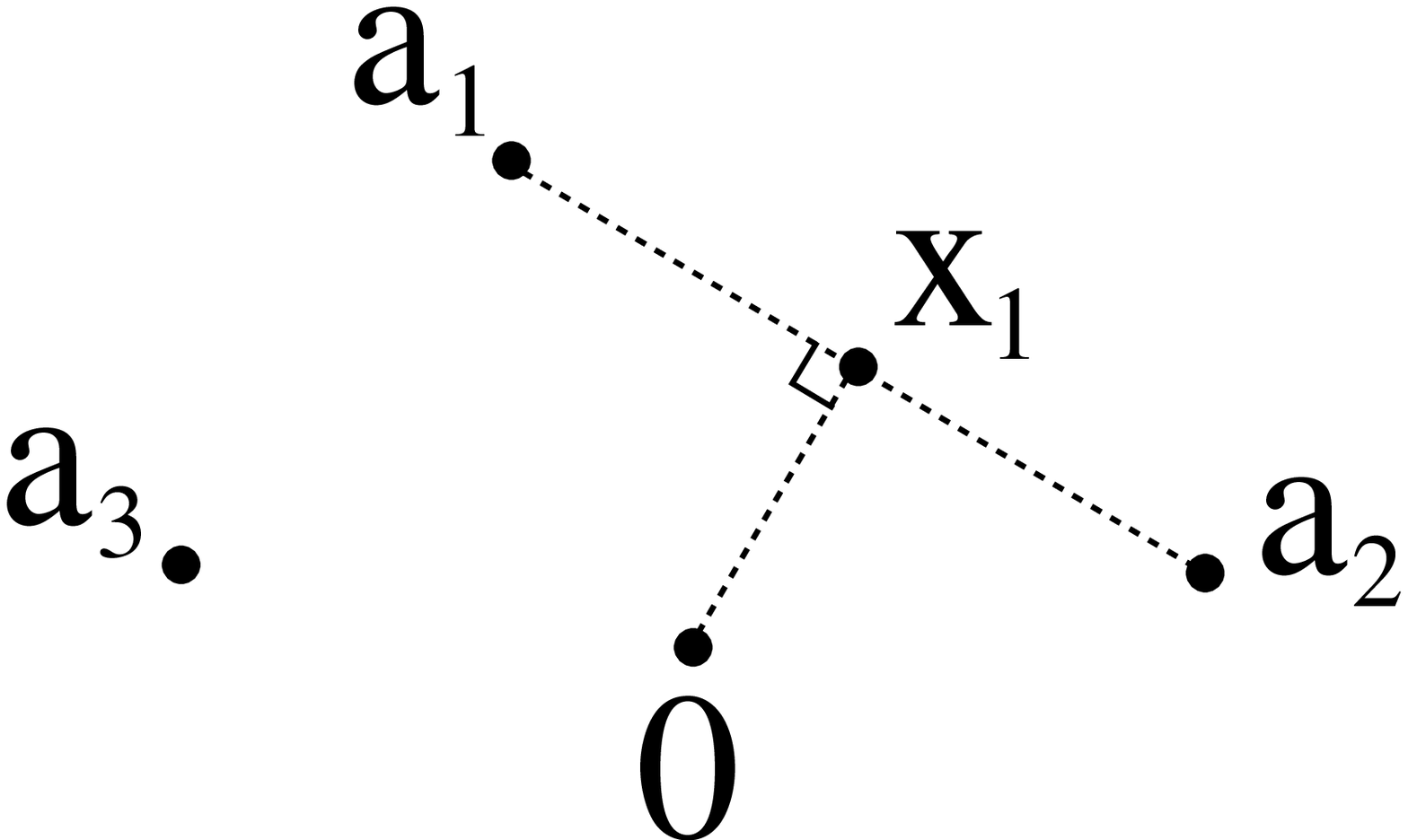} & \includegraphics[scale=0.2]{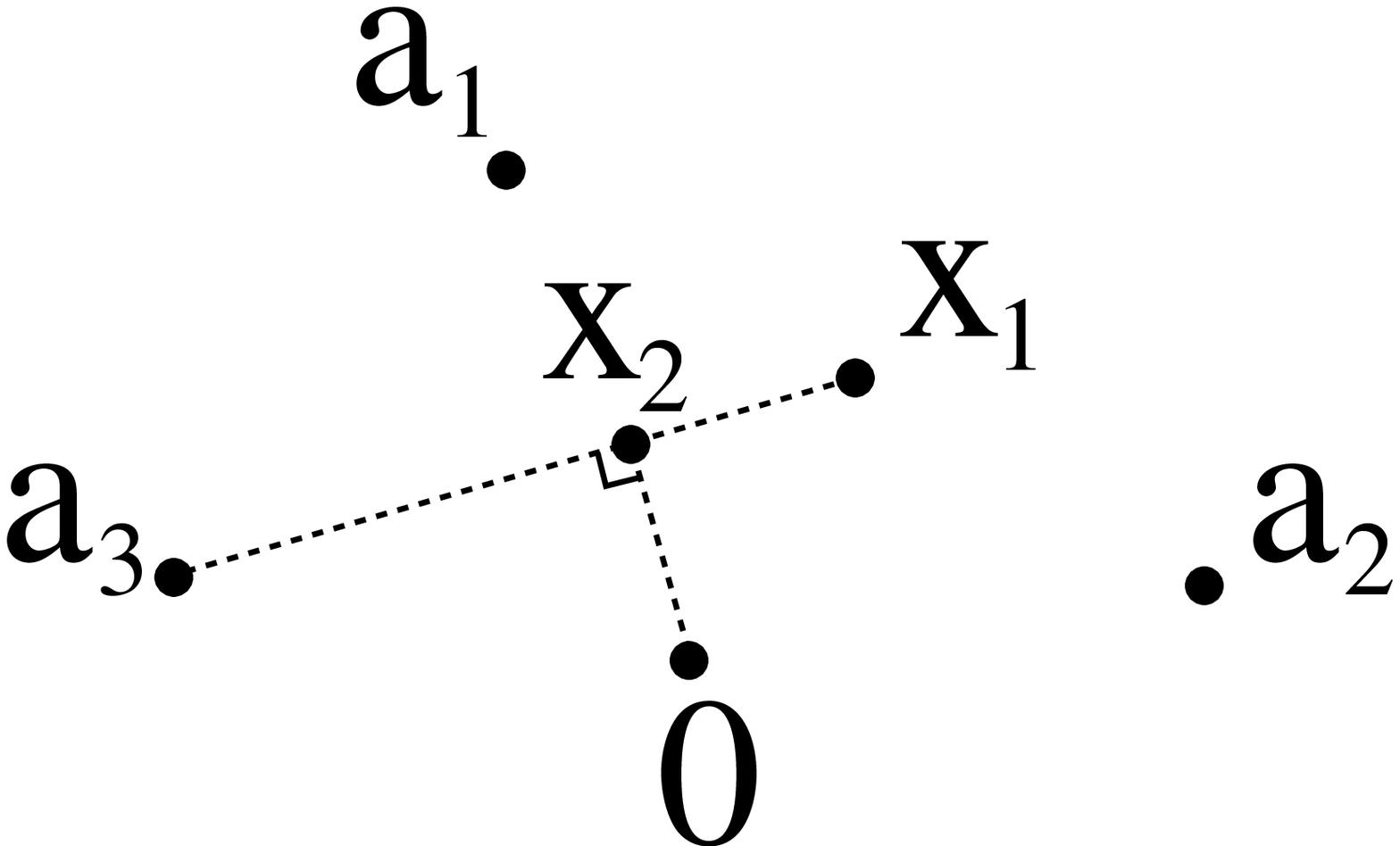}\tabularnewline
\hline 
\multicolumn{3}{c}{}\tabularnewline
\hline 
\multicolumn{3}{|l|}{Algorithm \ref{alg:enhanced-vN}}\tabularnewline
\hline 
\hline 
\includegraphics[scale=0.2]{vNm_01} & \includegraphics[scale=0.2]{vNm-02} & \includegraphics[scale=0.2]{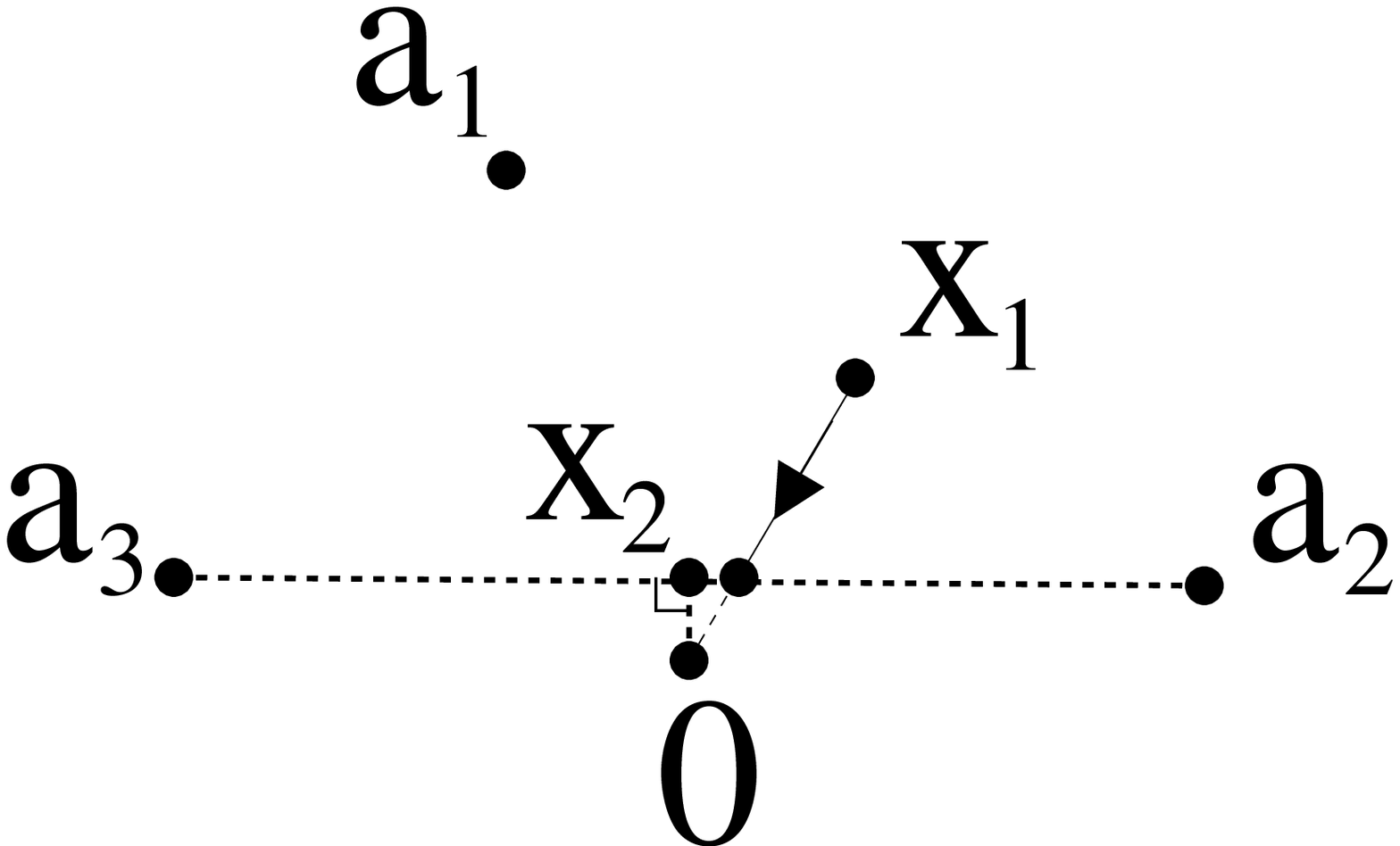}\tabularnewline
\hline 
\end{tabular}

\caption{\label{fig:sample-runs-QP-alg}The diagram on the top shows iterations
of the von Neumann algorithm, while the diagram on the bottom shows
iterations of Algorithm \ref{alg:enhanced-vN}. To find $x_{2}$ in
Algorithm \ref{alg:enhanced-vN}, the point $a_{3}$ is identified
in line 3. The algorithm projects onto $\conv\{a_{1},a_{2},a_{3}\}$
by moving from $x_{1}$ onto the line segment $[a_{2},a_{3}]$, then
moving along the line segment to get $x_{2}$. }
\end{figure}

We generalize Algorithm \ref{alg:enhanced-vN} to handle \eqref{eq:von-Neumann-gen}
in Algorithm \ref{alg:extended-primal-QP} below. A similar approach
was attempted in \cite{Rujan93} using the dual active set QP algorithm
\cite{Goldfarb_Idnani}. A dual quadratic programming algorithm is
considered to be better for general QP problems because there is no
need to find a feasible starting point. Since a feasible point to
\eqref{eq:von-Neumann-gen} is readily available, the primal approach
is not disadvantaged. More importantly, we feel that the primal QP
approach is better for \eqref{eq:von-Neumann-gen} because it works
with vectors in $\mathbb{R}^{m}$, whereas the dual approach works
with vectors in $\mathbb{R}^{n_{1}+n_{2}}$, and we expect $m\ll n_{1}+n_{2}$.
It is unclear whether this generalization is original or not, but
we feel that it is worthwhile to make a connection. Algorithm \ref{alg:extended-primal-QP}
can be pieced from the general structure of a primal active set QP
algorithm, and is similar to Algorithm \ref{alg:enhanced-vN}. Furthermore,
we will not elaborate Algorithm \ref{alg:extended-primal-QP}, nor
will the rest of this paper depend on Algorithm \ref{alg:extended-primal-QP},
so we shall be brief. 
\begin{algorithm}
\label{alg:extended-primal-QP}(Active set QP algorithm for \eqref{eq:von-Neumann-gen})
For $A\in\mathbb{R}^{m\times n_{1}}$ and $B\in\mathbb{R}^{m\times n_{2}}$,
this algorithm finds either a feasible pair $(x,z)$ of \eqref{eq:von-Neumann-gen}
satisfying $Ax=Bz$, or a $y\in\mathbb{R}^{m}$ such that 
\begin{equation}
\min_{j\in\{1,\dots,n_{1}\}}a_{j}^{T}y>\max_{k\in\{1,\dots,n_{2}\}}b_{k}^{T}y,\label{eq:separating-hyperplane}
\end{equation}
where $\{a_{j}\}_{j=1}^{n_{1}}$ are the columns of $A$ and $\{b_{k}\}_{k=1}^{n_{2}}$
are the columns of $B$. 

Choose $j\in\{1,\dots,n_{1}\}$, $k\in\{1,\dots,n_{2}\}$.

Set $J=\{j\}$, $K=\{k\}$ $x_{0}=e_{j}$, $y_{0}=Ax_{0}-Bz_{0}$
and $i=0$

Loop

$\stab$Find either some $j^{*}\in\{1,\dots,n_{1}\}$ such that $a_{j^{*}}^{T}y_{i}\leq\beta:=\max_{k\in K}b_{k}^{T}y_{i}$,

$\stab$$\stab$or some $k^{*}\in\{1,\dots,n_{2}\}$ such that $b_{k^{*}}^{T}y_{i}\geq\alpha:=\min_{j\in J}a_{j}^{T}y_{i}$. 

$\stab$If no such $j^{*}$ or $k^{*}$ exists, then $y_{i}$ solves
\eqref{eq:separating-hyperplane}, and we exit.

$\stab$\textbf{(Distance reduction loop)}

$\stab$Loop

$\stab$$\stab$If $j^{*}$ was found earlier

$\stab$$\stab$$\stab$Find closest points between 

$\stab$$\stab$$\stab$$\stab$$S_{1}:=\aff(\{a_{j}:j\in J\cup\{j^{*}\}\})$
and $S_{2}:=\aff(\{b_{k}:k\in K\})$, 

$\stab$$\stab$$\stab$$\stab$$\stab$say $s_{1}\in S_{1}$ and
$s_{2}\in S_{2}$. 

$\stab$$\stab$$\stab$$\stab$Write $s_{1}$ as $Ad_{1}$ and $s_{2}$
as $Bd_{2}$.

$\stab$$\stab$$\stab$Let
\begin{eqnarray*}
t_{1} & = & \min\left\{ \frac{(x_{i})_{j}}{-(d_{1}-x_{i})_{j}}:j\in J,(d_{1}-x_{i})_{j}<0\right\} \\
t_{2} & = & \min\left\{ \frac{(z_{i})_{k}}{-(d_{2}-z_{i})_{k}}:k\in K,(d_{2}-z_{i})_{k}<0\right\} \\
t & = & \min\{t_{1},t_{2},1\}
\end{eqnarray*}

$\stab$$\stab$$\stab$If $t=1$, then take $x_{i+1}\leftarrow d_{1}$
and $z_{i+1}\leftarrow d_{2}$, 

$\stab$$\stab$$\stab$$\stab$set $J\leftarrow J\cup\{j^{*}\}$
and exit loop.

$\stab$$\stab$$\stab$If $t\in(0,1)$, then $x_{i}\leftarrow x_{i}+t(d_{1}-x_{i})$
and $z_{i}\leftarrow z_{i}+t(d_{2}-z_{i})$, 

$\stab$$\stab$$\stab$$\stab$and drop the appropriate element in
either $j^{\prime}\in J$ or $k^{\prime}\in K$ 

$\stab$$\stab$$\stab$$\stab$such that $[x_{i}+t(d_{1}-x_{i})]_{j^{\prime}}=0$
or $[z_{i}+t(d_{2}-z_{i})]_{k^{\prime}}=0$.

$\stab$$\stab$end if

$\stab$$\stab$(The case where $k^{*}$ was found instead is similar)

$\stab$end loop.

$\stab$Set $y_{i}=Ax_{i}-Bz_{i}$

until $\|y_{i}\|$ small.
\end{algorithm}

\section{\label{sec:basic-analysis}Analysis of Algorithm \ref{alg:enhanced-vN}}

In this section, we prove some results of Algorithm \ref{alg:enhanced-vN}.
Theorem \ref{thm:fin-conv-1} gives conditions under which Algorithm
\ref{alg:enhanced-vN} terminates in finitely many iterations when
$0\in\intr(S)$, where 
\begin{equation}
S:=\{Ap:\bar{u}^{T}p=1,p\in K\}.\label{eq:set-S}
\end{equation}
We treat the case when $0$ lies in the boundary of $S$ and $S\subset\mathbb{R}^{2}$
in Subsection \ref{sub:bdry-case}, and show in Theorem \ref{thm:lin-conv-best}
that in such a case, we can expect linear convergence of $\{\|y_{i}\|\}_{i}$
to zero with a rate of at worst $1/\sqrt{2}$. 

We recall the convergence rates of the generalized von Neumann algorithm
for \eqref{eq:G_p_vN_pair}. 
\begin{rem}
(Convergence results from \cite{EpelmanFreund00}) The $\{\|y_{i}\|\}_{i}$
for the generalized von Neumann Algorithm for \eqref{eq:G_p_vN_pair}
is shown to be at worst linear when $0\in\intr(S)$, and at worst
sublinear with rate $O(\frac{1}{\sqrt{i}})$ when $0\in\partial S$
in \cite{EpelmanFreund00}. The second result can also be traced back
to \cite{Dantzig92_conv}. These results can be extended by copying
the proofs almost word for word for Algorithm \ref{alg:enhanced-vN}
as long as in line 8, $Ap_{i+1}$ and $y_{i}$ are contained in $\conv(\tilde{C})$.
In this paper, we shall concentrate on how we can get better rates
than those in \cite{EpelmanFreund00} when $\tilde{C}=C_{i+1}$. 
\end{rem}
We recall a easy result. A proof can be found in \cite{FreundVera99}
for example.
\begin{prop}
(Compactness of $S$) Suppose $\bar{u}\in\intr(K^{*})$. Then the
set $\{p:\bar{u}^{T}p=1,p\in K\}$ is compact. The set $S$ in \eqref{eq:set-S}
is compact as well.
\end{prop}
Our first result is the finite convergence of Algorithm \ref{alg:enhanced-vN}
if $0\in\intr(S)$.
\begin{thm}
\label{thm:fin-conv-1}(Finite convergence) Suppose Algorithm \ref{alg:enhanced-vN}
is used to solve \eqref{eq:G_p_vN_pair} for which \eqref{eq:von-Neumann-ineq-1}
is feasible. If $0\in\intr(S)$, where $S$ is as defined in \eqref{eq:set-S},
and the choices $p_{i+1}$ in line 3 and $C_{i+1}$ in line 7 are
chosen by \eqref{eq:new-p-i-1-formula} and
\[
C_{i+1}=\{Ap_{0},\dots,Ap_{i+1}\}
\]
 for all iterations $i$, then Algorithm \ref{alg:enhanced-vN} converges
in finitely many iterations.\end{thm}
\begin{proof}
Seeking a contradiction, suppose Algorithm \ref{alg:enhanced-vN}
runs indefinitely. Since $0\in\intr(S)$, let $\delta>0$ be such
that $\mathbb{B}(0,\delta)\subset S$, and let $M:=\max_{s\in S}\|s\|$.
Recall that $y_{i}=Ap_{i}$ for all $i$. We use induction to prove
that 
\begin{equation}
\angle[Ap_{j}]0[Ap_{k}]\geq\sin^{-1}(\delta/M)\mbox{ for all }0<j<k,\label{eq:fin-conv-indn-hyp}
\end{equation}
which leads to a contradiction because of the compactness of the unit
ball in $\mathbb{R}^{m}$. 

Suppose \eqref{eq:fin-conv-indn-hyp} is true for all $k\leq i$.
We show that \eqref{eq:fin-conv-indn-hyp} is true for $k=i+1$. Since
$y_{i}$ equals $P_{C_{i}}(0)$, where $C_{i}=\{Ap_{0},\dots,Ap_{i}\}$,
we have 
\begin{equation}
y_{i}^{T}Ap_{j}>0\mbox{ for all }j\in\{0,\dots,i\}.\label{eq:fin-conv-step-1}
\end{equation}
The point $p_{i+1}$ is chosen so that $Ap_{i+1}$ is a minimizer
of $\min\{[\frac{y_{i}}{\|y_{i}\|}]^{T}s:s\in S\}$. Note that since
$\mathbb{B}(0,\delta)\subset S$, $[\frac{y_{i}}{\|y_{i}\|}]^{T}Ap_{i+1}\leq-\delta$,
and $\|Ap_{i+1}\|\leq M$, we have 
\begin{equation}
\left[\frac{y_{i}}{\|y_{i}\|}\right]^{T}\frac{Ap_{i+1}}{\|Ap_{i+1}\|}\leq-\frac{\delta}{M}.\label{eq:fin-conv-step-2}
\end{equation}
Note that \eqref{eq:fin-conv-step-2} implies that $\angle[Ap_{i+1}]0v\geq\sin^{-1}(\delta/M)$
for all $v$ such that $y_{i}^{T}v\geq0$. Combining \eqref{eq:fin-conv-step-1}
and the induction hypothesis, we see that \eqref{eq:fin-conv-indn-hyp}
is true for $k=i+1$. This ends the proof of our result.
\end{proof}

\subsection{\label{sub:bdry-case}When $0\in\partial S$}

We now treat the case when $0$ lies in the boundary of $S$ (as defined
in \eqref{eq:set-S}) and $A\in\mathbb{R}^{2\times n}$ (i.e., $m=2$).
This setting implies $S\subset\mathbb{R}^{2}$, which in turn allows
for a detailed analysis.

If $p_{i+1}$ is chosen by \eqref{eq:new-p-i-1-formula}, then $Ap_{i+1}$
is a minimizer of $\min\{y_{i}^{T}s:s\in S\}$. Since $0\in\partial S$,
we look at $T_{S}(0)$, the tangent cone of $S$ at $0$, and the
case when $\dim(T_{S}(0))$ equals two. If $\dim(T_{S}(0))$ equals
to one instead, then $\dim(S)$ equals one, in which case $S$ is
a line segment. Once the end points of the line segment are identified,
we know all that we need about the set $S$. 

In view of the above discussions, we simplify Algorithm \ref{alg:enhanced-vN}
to the particular setting of interest where $m=2$ and $0\in\partial S$. 
\begin{algorithm}
\label{alg:simple-enh-vN}(Algorithm for system \eqref{eq:G_p_vN_pair})
For a compact convex set $S\subset\mathbb{R}^{2}$ containing $0$
on its boundary, this algorithm tries to find a sequence of iterates
$y\in\mathbb{R}^{2}$ converging to $0$. 

01$\quad$Set $y_{0}$ in the boundary of $S$ and $i=0$

02$\quad$Loop

03$\quad$$\formtab$Find $s_{i+1}$ such that $s_{i+1}$ is a minimizer
of $\min_{s\in S}\, y_{i}^{T}s.$

04$\quad$$\formtab$Let $C_{i+1}=\{y_{0},s_{1},\dots,s_{i+1}\}$,
$y_{i+1}:=P_{\scriptsize\conv(C_{i+1})}(0)$ and $i\leftarrow i+1$

05$\quad$until $\|y_{i}\|$ small.
\end{algorithm}
Note that lines 3-5 of Algorithm \ref{alg:enhanced-vN} accommodate
for the cases when $0\in\intr(S)$, $0\in\partial S$ and $0\notin S$.
Since we only wish to study the case where $0\in\partial S$, we took
out the corresponding lines in Algorithm \ref{alg:simple-enh-vN}. 

We also enforced that $y_{0}$ lies on the boundary of $S$ to simplify
our analysis.

In line 4 of Algorithm \ref{alg:simple-enh-vN}, we do not try to
reduce the size of $C_{i+1}$. The next result shows that since $m=2$,
there is no need to reduce the size of $C_{i+1}$. For 2 points $\alpha,\beta\in\mathbb{R}^{2}$,
we let $(\alpha,\beta)$ denote the set 
\[
(\alpha,\beta)=\{t\alpha+(1-t)\beta:t\in(0,1)\}.
\]

\begin{thm}
\label{thm:bisec-enhanced-vN}(Bisection behavior of Algorithm \ref{alg:simple-enh-vN})
In Algorithm \ref{alg:simple-enh-vN}, suppose $y_{0}\neq0$ and $\dim(S)=2$.
Then for each $i\geq0$, 
\begin{enumerate}
\item Either $s_{1}=0$ or $y_{1}\in(y_{0},s_{1})$. 
\item Suppose $s_{i}\neq0$, $y_{i}\neq0$, and that $y_{i}\in(\bar{c}_{i},s_{i})$
for some $\bar{c}_{i}\in C_{i-1}$. \\
Then either $y_{i+1}\in(\bar{c}_{i},s_{i+1})$, $y_{i+1}\in(s_{i},s_{i+1})$,
or $s_{i+1}=0$.
\end{enumerate}
\end{thm}
\begin{figure}[!h]
\includegraphics[scale=0.5]{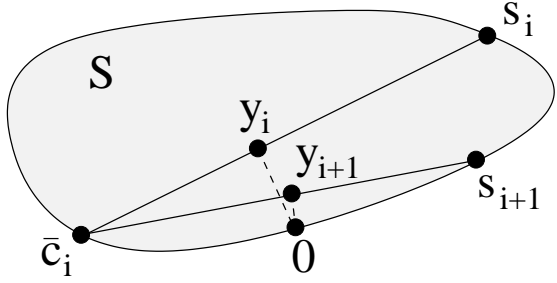}\caption{\label{fig:illustrate-bisect}Illustration of Theorem \ref{thm:bisec-enhanced-vN}.}
\end{figure}

\begin{proof}
If at any point $s_{i+1}=0$, then $y_{i+1}=P_{C_{i+1}}(0)=0$, resulting
in the termination of Algorithm \ref{alg:simple-enh-vN}. We shall
rule this case out to simplify our proof. 

We first prove (1). For $i=0$, $y_{1}=P_{\scriptsize\conv\{y_{0},s_{1}\}}(0)$.
Since $s_{1}$ is chosen to be a minimizer of $\min_{s\in S}y_{0}^{T}s$,
we have $y_{0}^{T}s_{1}\leq y_{0}^{T}0=0$. Since $s_{1}\neq0$ and
$y_{0}\neq0$, this means that $\angle y_{0}0s_{1}\geq\pi/2$, which
implies that $\angle0y_{0}s_{1}<\pi/2$ and $\angle0s_{1}y_{0}<\pi/2$.
So $y_{1}=P_{\scriptsize\conv\{y_{0},s_{1}\}}(0)$ must lie in the
set $(y_{0},s_{1})$. 

\textbf{}The projection of $0$ onto the polyhedron $\conv(C_{i})$
must land on a face of the polyhedron. If such a face is 2-dimensional,
then this means that $0$ lies in the (relative) interior of $\conv(C_{i})$,
but this would imply that $0$ lies in the interior of $S$ as $\intr\,\conv(C_{i})\subset\intr(S)$.
If the face if 0-dimensional, this means that $y_{i}=P_{\scriptsize\conv(C_{i})}(0)$
is a point in $C_{i}$. The other possibility is that the face is
1-dimensional, which corresponds to $y_{i}=P_{\scriptsize\conv(C_{i})}(0)$
being in $(c_{1},c_{2})$, where $c_{1},c_{2}$ are distinct elements
in $C_{i}$.

We prove (2) by induction. Statement (1) shows that the base case
holds. Suppose our claim is true for $i=i^{*}$. Then $y_{i^{*}}=(\bar{c}_{i^{*}},s_{i^{*}})$
for some $\bar{c}_{i^{*}}\in C_{i^{*}}$. Now, $y_{i^{*}}=P_{\scriptsize\conv C_{i^{*}}}(0)$
implies that 
\begin{equation}
y_{i^{*}}^{T}c\geq y_{i^{*}}^{T}y_{i^{*}}>0\mbox{ for all }c\in\conv(C_{i^{*}}).\label{eq:y-c-geq-0-conv-C}
\end{equation}

\textbf{Claim 1: $y_{i^{*}+1}$ cannot be a point in $C_{i^{*}+1}$. }

We take a look at $y_{i^{*}+1}=P_{\scriptsize\conv(C_{i^{*}+1})}(0)$.
If $y_{i^{*}+1}$ is some point in $C_{i^{*}+1}$, then the possibilities
are that $y_{i^{*}+1}\in C_{i^{*}}$ or $y_{i^{*}+1}=s_{i^{*}+1}$.
If $y_{i^{*}+1}\in C_{i^{*}}$, then note that $\conv(C_{i^{*}})\subset\conv(C_{i^{*}+1})$,
so $y_{i^{*}}=P_{\scriptsize\conv(C_{i^{*}})}$ must be a point in
$C_{i^{*}}$ as well, but this is ruled out by the induction hypothesis.
We now rule out $y_{i^{*}+1}=s_{i^{*}+1}$. Now $s_{i^{*}+1}\in\conv(C_{i^{*}+1})$
and $y_{i^{*}}\in\conv(C_{i^{*}})\subset\conv(C_{i^{*}+1})$. Recall
that $s_{i^{*}+1}$ is chosen to be a minimizer of $\min_{s\in S}y_{i^{*}}^{T}s$,
so 
\begin{equation}
y_{i^{*}}^{T}s_{i^{*}+1}\leq y_{i^{*}}^{T}0=0,\label{eq:y-s-1-geq-0}
\end{equation}
or $\angle y_{i^{*}}0s_{i^{*}+1}\geq\pi/2$. Since $y_{i^{*}}\neq0$
and $s_{i^{*}+1}\neq0$, we have $\angle y_{i^{*}}s_{i^{*}+1}0<\pi/2$
and $\angle s_{i^{*}+1}y_{i^{*}}0<\pi/2$, so 
\[
d(0,\{s_{i^{*}+1},y_{i^{*}}\})>d(0,\conv\{s_{i^{*}+1},y_{i^{*}}\})\geq d(0,\conv(C_{i^{*}+1})).
\]
Thus $y_{i^{*}+1}$ cannot be a point in $C_{i^{*}+1}$. 

\textbf{Claim 2: If $y_{i^{*}}\in(\bar{c}_{i^{*}},s_{i^{*}})$, then
either $y_{i^{*}+1}\in(\bar{c}_{i^{*}},s_{i^{*}+1})$ or $y_{i^{*}+1}\in(s_{i^{*}},s_{i^{*}+1})$ }

If $y_{i^{*}+1}$ lies in some line segment $(c_{1},c_{2})$, where
$c_{1}$ and $c_{2}$ are distinct elements in $C_{i^{*}}$, then
\[
d(0,C_{i^{*}})=d(0,y_{i^{*}})>d\big(0,(s_{i^{*}+1},y_{i^{*}})\big)\geq d(0,C_{i^{*}+1})=d(0,(c_{1},c_{2}))\geq d(0,C_{i^{*}}),
\]
which is absurd. Thus $y_{i^{*}+1}$ must lie in the segment $(c,s_{i^{*}+1})$
for some $c\in C_{i^{*}}$. We need to prove that $c$ can only be
either $\bar{c}_{i^{*}}$ or $s_{i^{*}}$. 

Since $y_{i^{*}}\in(\bar{c}_{i^{*}},s_{i*})$ and $\bar{c}_{i^{*}}$
and $s_{i^{*}}$ both lie on the boundary of $S$, the line $\aff(\{\bar{c}_{i^{*}},s_{i^{*}}\})$
is a supporting hyperplane of $\conv(C_{i^{*}})$ at $y_{i^{*}}$.
Take any $c\in C_{i^{*}}\backslash\{\bar{c}_{i^{*}},s_{i^{*}}\}$.
Since $S\subset\mathbb{R}^{2}$, we make use of \eqref{eq:y-s-1-geq-0}
and \eqref{eq:y-c-geq-0-conv-C} to see that the line segment $[c,s_{i^{*}+1}]$
has to intersect somewhere in the line segment $[\bar{c}_{i^{*}},s_{i^{*}}]$.
By working out the possibilities in $\mathbb{R}^{2}$, we see that
\[
\min\big(d(0,[\bar{c}_{i^{*}},s_{i^{*}+1}]),d(0,[s_{i^{*}},s_{i^{*}+1}])\big)\leq d(0,[c,s_{i^{*}+1}]).
\]
Thus $y_{i^{*}+1}$ has to be in $(\bar{c}_{i^{*}},s_{i^{*}+1})$
or $(s_{i^{*}},s_{i^{*}+1})$. 
\end{proof}
The consequence of Theorem \ref{thm:bisec-enhanced-vN} is that when
$\dim(S)=2$, there is no need to revisit dropped boundary points
of $S$ in the active set QP algorithm to project onto the convex
hull of an increasing set of points $C_{i}$. 

Another way to interpret Theorem \ref{thm:bisec-enhanced-vN} is as
follows. The boundary of $S$ is homeomorphic to the sphere $\{x\in\mathbb{R}^{2},|x|=1\}$.
Algorithm \ref{alg:simple-enh-vN} is a bisection strategy. In iteration
$i$ when $y_{i}\in(\bar{c}_{i},s_{i})$ as in the notation of Theorem
\ref{thm:bisec-enhanced-vN}, Algorithm \ref{alg:simple-enh-vN} identifies
that $0$ lies on the path along the boundary from $\bar{c}_{i}$
to $s_{i}$. After the next iteration, either $y_{i+1}\in(\bar{c}_{i},s_{i+1})$
or $y_{i+1}\in(s_{i+1},s_{i})$. This means that Algorithm \ref{alg:simple-enh-vN}
has found that $0$ lies along the path along the boundary of $S$
from $s_{i+1}$ to either $\bar{c}_{i}$ or $s_{i}$. Notice that
even if $0$ were very close to $\bar{c}_{i}$ (or $s_{i}$ instead)
for example, the next point $s_{i+1}$ depends only on the geometry
of $S$ and not on the position of $y_{i}$. We shall see in Proposition
\ref{prop:any-sequence} that the ratio between $\|s_{i+1}-\bar{c}_{i}\|$
and $\|s_{i+1}-s_{i}\|$ can be arbitrarily large or small.

As a consequence of Theorem \ref{thm:bisec-enhanced-vN}, we prove
that the convergence of $\{\|y_{i}\|\}_{i}$ to zero is at least linear.
\begin{thm}
\label{thm:lin-conv-best}(Linear convergence of $\{\|y_{i}\|\}_{i}$
in Algorithm \ref{alg:simple-enh-vN}) The convergence of $\{\|y_{i}\|\}_{i}$
in Algorithm \ref{alg:simple-enh-vN} to zero is at worst linear with
constant $1/\sqrt{2}$. \end{thm}
\begin{proof}
Let $w_{i}=\|\bar{c}_{i}-s_{i}\|$. We assume that the convergence
of $\{\|y_{i}\|\}_{i}$ to zero is not finite. We deal with the easier
case first.

\textbf{Case 1: There is some $K>0$ such that if $i^{*}$ is large
enough, $i>i^{*}$ and $\frac{w_{i}}{w_{i^{*}}}\leq2^{-(i-i^{*})/2}$,
then $\|y_{i}\|\leq Kw_{i}\leq2^{-(i-i^{*})/2}Kw_{i^{*}}$.}

Recall $S$ is a convex set in $\mathbb{R}^{2}$ and $0$ lies in
the path from $\bar{c}_{i}$ to $s_{i}$ along $\partial S$. If $\frac{w_{i}}{w_{i^{*}}}\leq2^{-(i-i^{*})/2}$,
we must have $w_{i}\to0$, so $\bar{c}_{i}\to0$ and $s_{i}\to0$.
The limit $\lim_{i\to\infty}\angle\bar{c}_{i}0s_{i}$ exists as $\{\angle\bar{c}_{i}0s_{i}\}_{i}$
is nondecreasing and equals
\[
\lim_{i\to\infty}\angle\bar{c}_{i}0s_{i}=\max\{\angle v_{1}0v_{2}:v_{1},v_{2}\in T_{S}(0)\},
\]
which is finite. Let the limit above be $\theta>0$. We can use elementary
geometry to figure that $\|y_{i}\|$, which is also $d(0,\aff(\{\bar{c}_{i},s_{i}\}))$,
equals 
\[
\|y_{i}\|=\frac{w_{i}\sin\angle\bar{c}_{i}s_{i}0\sin\angle s_{i}\bar{c}_{i}0}{\sin\angle\bar{c}_{i}0s_{i}},
\]
Moreover, 
\begin{eqnarray*}
\limsup_{i\to\infty}\frac{\sin\angle\bar{c}_{i}s_{i}0\sin\angle s_{i}\bar{c}_{i}0}{\sin\angle\bar{c}_{i}0s_{i}} & \leq & \lim_{i\to\infty}\frac{[\sin(\frac{1}{2}[\pi-\angle\bar{c}_{i}0s_{i}])]^{2}}{\sin\angle\bar{c}_{i}0s_{i}}\\
 & = & \begin{cases}
\frac{[\sin(\frac{1}{2}[\pi-\theta])]^{2}}{\sin\theta} & \mbox{ if }\theta<\pi\\
0 & \mbox{ if }\theta=\pi.
\end{cases}
\end{eqnarray*}
There is some $K>0$ such that if $i^{*}$ is large enough and $i>i^{*}$,
then $\|y_{i}\|\leq Kw_{i}$. The remaining inequality is easy, and
this ends our proof for case 1.

Let $q_{i}$ be a point such that $\aff(\{s_{i},q_{i}\})$ is a supporting
hyperplane of $S$ at $s_{i}$, and $q_{i}$ lies on the same side
of $\aff(\{\bar{c}_{i},s_{i}\})$ as $0$. Similarly, let $\bar{q}_{i}$
be a point such that $\aff(\{\bar{c}_{i},\bar{q}_{i}\})$ is a supporting
hyperplane of $S$ at $\bar{c}_{i}$ and $\bar{q}_{i}$ lies on the
same side of $\aff(\{\bar{c}_{i},s_{i}\})$ as $0$. See Figure \ref{fig:pf-lin-conv}.

\textbf{Claim 1:} \textbf{If $\angle\bar{c}_{i}s_{i}q_{i}$ and $\angle s_{i}\bar{c}_{i}\bar{q}_{i}$
are acute, then both $\angle\bar{c}_{i+1}s_{i+1}q_{i+1}$ and $\angle s_{i+1}\bar{c}_{i+1}\bar{q}_{i+1}$}
\textbf{are acute.}

Without loss of generality, assume that $\bar{c}_{i+1}=\bar{c}_{i}$.
(The other possibility of $\bar{c}_{i+1}=s_{i}$ is similar.) One
can see from Figure \ref{fig:pf-lin-conv} that $\bar{q}_{i+1}$ can
be taken to be $\bar{q}_{i}$. The $q_{i+1}$ is also easy to choose.
One can see that $\angle\bar{c}_{i+1}s_{i+1}q_{i+1}$ and $\angle s_{i+1}\bar{c}_{i+1}\bar{q}_{i+1}$
are both acute as claimed.

\begin{figure}
\begin{tabular}{|c|c|}
\hline 
\includegraphics[scale=0.4]{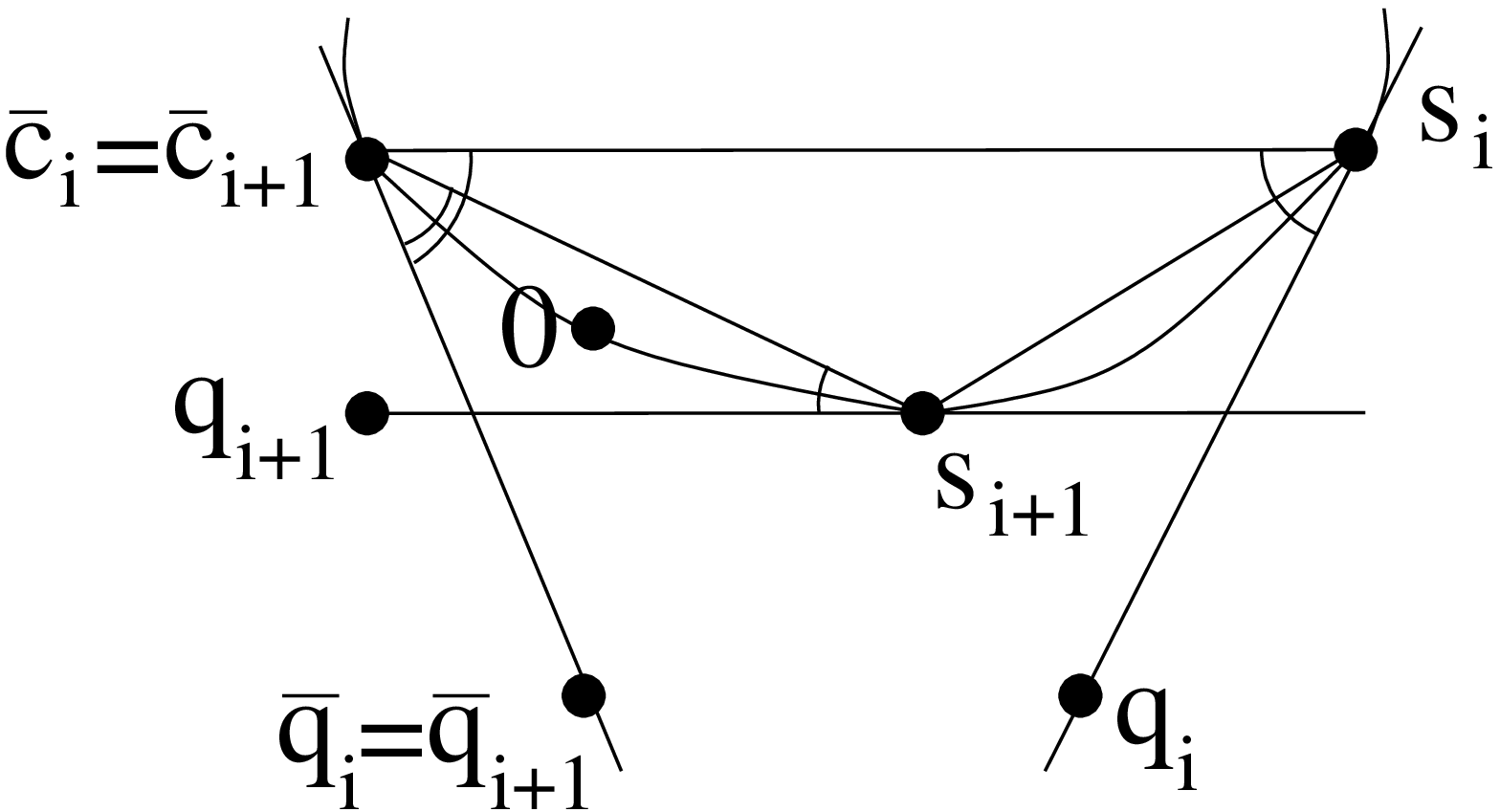} & \includegraphics[scale=0.4]{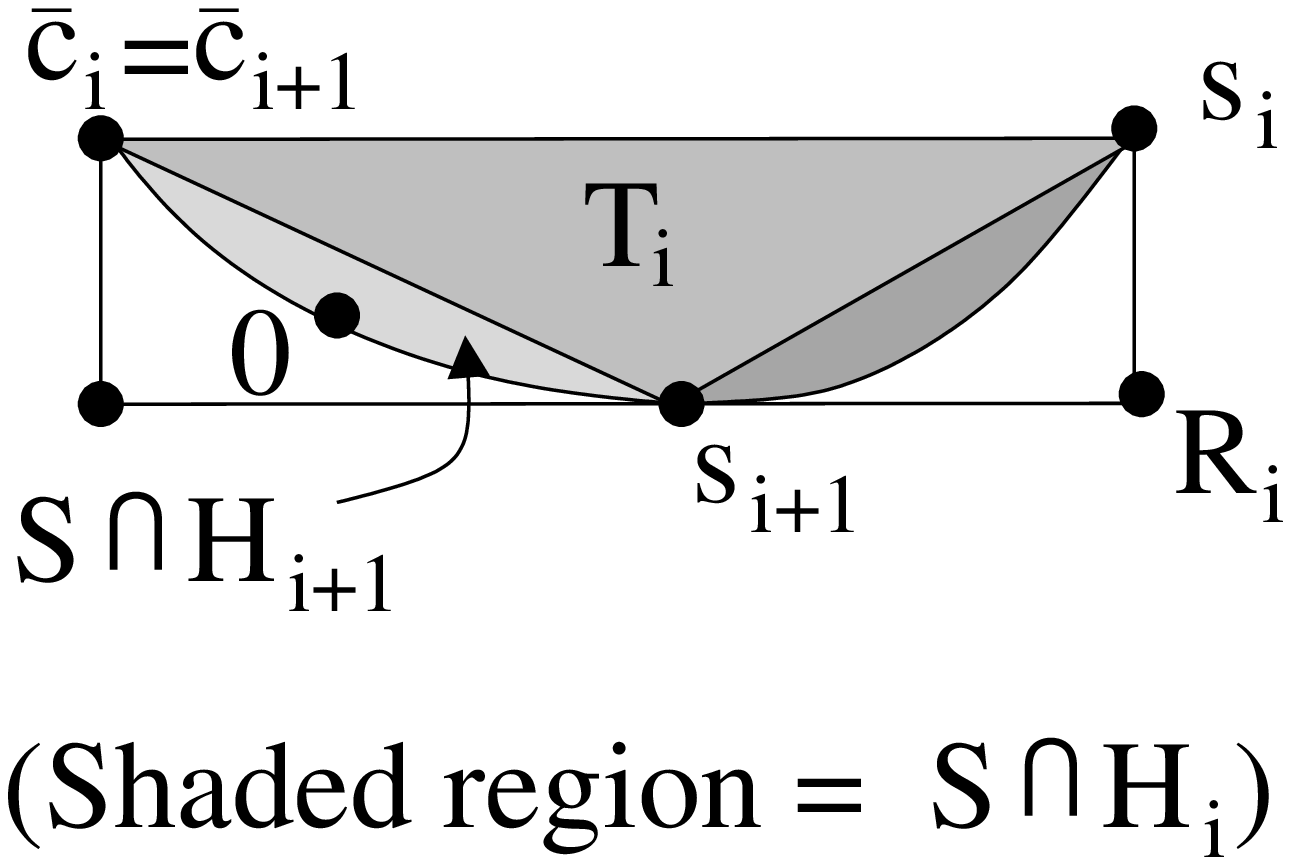}\tabularnewline
\hline 
\end{tabular}

\caption{\label{fig:pf-lin-conv}The diagram on the left is that of the proof
of Claim 1 in Theorem \ref{thm:lin-conv-best}, while the diagram
on the right is that of the proof of Claim 3 in the same theorem.}

\end{figure}

\textbf{Claim 2: For $i$ large enough, both $\angle\bar{c}_{i}s_{i}q_{i}$
and $\angle s_{i}\bar{c}_{i}\bar{q}_{i}$ are acute.}

Note that $\bar{c}_{1}=y_{0}$. One can easily see that $\angle y_{0}s_{1}q_{1}$
is acute. However, the angle $\angle s_{1}y_{0}\bar{q}_{1}$ is not
necessarily acute. If every point on the line segment $[y_{0},0]$
is on the boundary of $S$, then we can choose $\bar{q}_{1}$ so that
$\angle s_{1}y_{0}\bar{q}_{1}$ is acute. 

Consider the following statement:
\begin{itemize}
\item [(*)] Unless every point on the line segment $[y_{0},0]$ is on the
boundary of $S$ (which was already treated in the previous paragraph),
eventually $\bar{c}_{i}\neq y_{0}$ for all $i$ large enough.
\end{itemize}
We now show that ({*}) implies our claim at hand. Suppose ({*}) is
true. Let $i^{*}$ be the smallest $i$ such that $\bar{c}_{i}\neq y_{0}$.
This would mean that $\bar{c}_{i^{*}}=s_{i^{*}-1}$. The angle $\angle s_{i^{*}}\bar{c}_{i^{*}}\bar{q}_{i^{*}}$
can be checked to be acute, and so would $\angle\bar{c}_{i^{*}}s_{i^{*}}q_{i^{*}}$. 

We now prove ({*}) by contradiction. Suppose $\bar{c}_{i}=y_{0}$
for all $i$. The points $\{s_{i}\}_{i}$ trace a path along $\partial S$
getting closer to $0$. Let $s^{*}:=\lim_{i\to\infty}s_{i}$ and 
\[
y^{*}:=\lim_{i\to\infty}\frac{P_{\scriptsize\conv(\{y_{0},s_{i}\})}(0)}{\|P_{\scriptsize\conv(\{y_{0},s_{i}\})}(0)\|}.
\]
If $s^{*}\neq0$, we can see that $y^{*}=\frac{P_{\scriptsize\conv(\{y_{0},s^{*}\})}(0)}{\|P_{\scriptsize\conv(\{y_{0},s^{*}\})}(0)\|}$.
If $s^{*}=0$, then $y^{*}$ is the vector perpendicular to $\aff(\{y_{0},0\})$
such that $s_{1}^{T}y^{*}>0$. Since all points in the line segment
$(y_{0},0)$ lie in $\intr(S)$, all points in the line segment $(y_{0},s^{*})$
also lie in $\intr(S)$. Any minimizer of $\min\{s^{T}y^{*}:s\in S\}$
lies on the path along the boundary of $S$ between $s^{*}$ and $y_{0}$.
So if $s_{i}$ were sufficiently close to $s^{*}$, $s_{i+1}$ would
be forced to be on the boundary of $S$ between $s^{*}$ and $y_{0}$
as well. This contradicts the assumption that $s^{*}=\lim_{i\to\infty}s_{i}$,
ending the proof of the claim.

Let $A_{i}$ be the area of $S\cap H_{i}$, where $H_{i}$ is the
halfspace with boundary $\aff(\{\bar{c}_{i},s_{i}\})$ containing
$0$. See Figure \ref{fig:pf-lin-conv}.

\textbf{Claim 3: $2A_{i+1}\leq A_{i}$ if $i$ is large enough so
that claim 2 holds}

Let the triangle $\conv(\{\bar{c}_{i},s_{i},s_{i+1}\})$ be $T_{i}$.
See Figure \ref{fig:pf-lin-conv}. If $i$ is large enough so that
claim 2 holds, then the set $S\cap H_{i}$ is bounded by four lines:
the line $\aff(\{\bar{c}_{i},s_{i}\})$, the line parallel to $\aff(\{\bar{c}_{i},s_{i}\})$
through $s_{i+1}$, and the lines perpendicular to $\aff(\{\bar{c}_{i},s_{i}\})$
through $\bar{c}_{i}$ and $s_{i}$. The rectangle formed, which we
call $R_{i}$, has twice the area of $T_{i}$. It is clear that $[S\cap H_{i+1}]\cup T_{i}\subset S\cap H_{i}\subset R_{i}$,
which implies $A_{i+1}+\mbox{area}(T_{i})\leq A_{i}$. Also, $S\cap H_{i+1}\subset R_{i}\backslash T_{i}$,
which implies $A_{i+1}\leq\mbox{area}(T_{i})$. Thus $2A_{i+1}\leq A_{i}$,
which is the conclusion we seek.

We now consider the second case.

\textbf{Case 2: If $i^{*}$ is large enough, $i>i^{*}$ and $\frac{w_{i}}{w_{i^{*}}}\geq2^{-(i-i^{*})/2}$,
then $\|y_{i}\|\leq\frac{2A_{i}}{w_{i}}\leq\frac{2^{-i+i^{*}+1}A_{i^{*}}}{2^{-(i-i^{*})/2}w_{i^{*}}}=2^{-(i-i^{*})/2}\frac{2A_{i^{*}}}{w_{i^{*}}}$.}

It is clear from elementary geometry that 
\[
\|y_{i}\|w_{i}\leq d(s_{i+1},\aff(\{\bar{c}_{i},s_{i}\}))w_{i}=\mbox{area}(R_{i})=2\mbox{area}(T_{i})\leq2A_{i},
\]
or in other words $\|y_{i}\|\leq\frac{2A_{i}}{w_{i}}$. Claim 3 implies
that $A_{i}\leq2^{-i+i^{*}}A_{i^{*}}$ if $i>i^{*}$ and $i^{*}$
is large enough. This ends the proof of our result for case 2.

Putting together the two cases gives us the result at hand.
\end{proof}

\section{\label{sec:More-analysis}More on Algorithm \ref{alg:simple-enh-vN}}

In this section, we continue from the developments in Section \ref{sec:basic-analysis}
and elaborate on the behavior of Algorithm \ref{alg:simple-enh-vN}
by using an epigraphical and subdifferential analysis.

When $\dim(S)=2$, the intersection of $\mathbb{B}(0,\delta)\cap S$
is, up to a rotation, the intersection of a compact convex set and
the epigraph of some convex function, say $f$. This is described
in Figure \ref{fig:epiLipschitzian}. \marginpar{Remember to put this figure and the subsequent figure to latex file
to make the latex pictures render properly.}(The set $S$ is said to be epi-Lipschitzian \cite{Rockafellar79_epiL}
at $0$ in the sense of variational analysis. For more information,
see \cite{Cla83,Mor06,RW98} for example.)

\begin{figure}[!h]
\scalebox{0.45}[0.45]{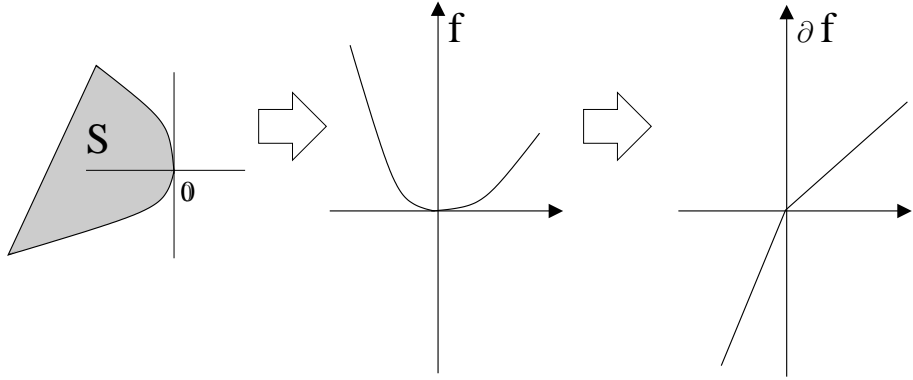}

\caption{\label{fig:epiLipschitzian} The set $S$ is epi-Lipschitzian at $0$.
Therefore, we can rotate the set $S$ so that locally at $0$, it
is the epigraph of some convex function $f$. We will need to use
the subdifferential mapping $\partial f(\cdot)$ for our later analysis.}
\end{figure}

We look at the graph of subdifferential $\partial f:\mathbb{R}\rightrightarrows\mathbb{R}$,
where ``$\rightrightarrows$'' signifies that $\partial f(\cdot)$
is a set-valued map, or in other words, $\partial f(x)$ is in general
a subset of $\mathbb{R}$. Since $f(\cdot)$ is convex, it is well-known
that the subdifferential mapping $\partial f(\cdot)$ is \emph{monotone},
i.e., if $v_{1}\in\partial f(x_{1})$, $v_{2}\in\partial f(x_{2})$
and $x_{1}\leq x_{2}$, then $v_{1}\leq v_{2}$. In view of monotonicity,
the points of discontinuity of $\partial f(\cdot)$ on an interval
is of measure zero and the function $\partial f(\cdot)$ is integrable,
i.e., 
\[
\int_{\alpha}^{\beta}\partial f(x)dx=f(\beta)-f(\alpha).
\]
We now state an algorithm expressed in terms of $f$ and $\partial f$,
and show its relationship with Algorithm \ref{alg:enhanced-vN}.
\begin{algorithm}
\label{alg:fn-bracketing}(A bracketing algorithm) For $a_{0},b_{0}>0$,
let $f:[-a_{0},b_{0}]\to\mathbb{R}$ be a convex function with a minimizer
at $0$. We want to find a minimizer of $f(\cdot)$ with the following
steps.

01$\stab$Start with $i=0$

02$\stab$Loop

03$\stab$Find a point in $[\partial f]^{-1}(\frac{f(b_{i})-f(-a_{i})}{a_{i}+b_{i}})$,
say $c_{i}$, which lies in the interval $[-a_{i},b_{i}]$. 

04$\stab$$\stab$If $c_{i}<0$, then $a_{i+1}\leftarrow-c_{i}$ and
$b_{i+1}\leftarrow b_{i}$.

05$\stab$$\stab$If $c_{i}>0$, then $a_{i+1}\leftarrow a_{i}$ and
$b_{i+1}\leftarrow c_{i}$.

06$\stab$$\stab$If $a_{i+1}+b_{i+1}$ is sufficiently small or $c_{i}=0$,
then end algorithm.

07$\stab$$\stab$$i\leftarrow i+1$

08$\stab$end loop
\end{algorithm}
At each step of Algorithm \ref{alg:fn-bracketing}, we find $a_{i}$
and $b_{i}$ such that $0\in(-a_{i},b_{i})$. Each iteration improves
either the left or right end point.

In line 3 of Algorithm \ref{alg:simple-enh-vN}, we find a minimizer
of $\min_{s\in S}y_{i}^{T}s$, where $y_{i}$ is the projection of
$0$ onto $C_{i}$. In line 3 of Algorithm \ref{alg:fn-bracketing},
we find a minimizer of $x\mapsto f(x)-[\frac{f(b_{i})-f(-a_{i})}{a_{i}+b_{i}}]^{T}x$
by finding a point $c_{i}$ such that $\frac{f(b_{i})-f(-a_{i})}{a_{i}+b_{i}}\in\partial f(c_{i})$.
It is clear to see that line 3 of both algorithms are equivalent. 

The following result shows the basic convergence of Algorithm \ref{alg:fn-bracketing}.
\begin{thm}
\label{thm:interval-analysis}(Basic convergence of Algorithm \ref{alg:fn-bracketing})
Let $\bar{a}$ and $\bar{b}$ be two positive numbers, and $f:[-\bar{a},\bar{b}]\to\mathbb{R}$.
Suppose $a^{\prime}\in[0,\bar{a}]$ and $b^{\prime}\in[0,\bar{b}]$
are such that 
\[
f(x)\begin{cases}
=0 & \mbox{if }x\in[-a^{\prime},b^{\prime}]\\
>0 & \mbox{otherwise}.
\end{cases}
\]
 Then the iterates $\{a_{i}\}_{i}$ and $\{b_{i}\}_{i}$ of Algorithm
\ref{alg:fn-bracketing} are such that $\{a_{i}\}_{i}$ and $\{b_{i}\}_{i}$
are non-increasing sequences such that for each $i$, either $a_{i+1}<a_{i}$
or $b_{i+1}<b_{i}$. Furthermore, one of these possibilities happen
\begin{enumerate}
\item Algorithm \ref{alg:fn-bracketing} finds a point in $[\partial f]^{-1}(0)$.
(i.e., a minimizer of $f$ is found.)
\item $\lim_{i\to\infty}a_{i}=a^{\prime}$ and $\lim_{i\to\infty}b_{i}=b^{\prime}$.

\end{enumerate}
\end{thm}
\begin{proof}
Assume that Algorithm \ref{alg:fn-bracketing} does not encounter
a point in $[\partial f]^{-1}(0)$. We try to show that only case
(2) can happen.

When $a^{\prime}=\bar{a}$ and $b^{\prime}=\bar{b}$, then $\partial f(x)=\{0\}$
for all $x\in(-\bar{a},\bar{b})$, so case (1) must happen. When $a^{\prime}<\bar{a}$
and $b^{\prime}=\bar{b}$, Algorithm \ref{alg:fn-bracketing} applied
to $f(\cdot)$ gives equivalent iterates as Algorithm \ref{alg:fn-bracketing}
applied to $f(-\cdot)$, where the $a$'s and $b$'s swap roles, reducing
to the case where $a^{\prime}=\bar{a}$ and $b^{\prime}<\bar{b}$.
We look at two cases from here onwards. 

\textbf{Case A: $a^{\prime}=\bar{a}$ and $b^{\prime}<\bar{b}$.}

It is obvious that $\lim_{i\to\infty}a_{i}=\bar{a}=a^{\prime}$. If
case (1) is not encountered, then $b_{i}>b^{\prime}$ for all $i$.
We prove that $\lim_{i\to\infty}b_{i}=b^{\prime}$. Let $\partial f(b_{i})=[s_{i,3},s_{i,4}]$
for all $i$. Now, 
\[
s_{i}^{\prime}:=\frac{1}{\bar{a}+b_{i}}\int_{-\bar{a}}^{b_{i}}\partial f(x)dx\leq\frac{1}{\bar{a}+b_{i}}\int_{0}^{b_{i}}s_{i,3}dx\leq\frac{\bar{b}s_{i,3}}{\bar{a}+\bar{b}}.
\]
It is clear to see that $s_{i}^{\prime}\in(0,\frac{\bar{b}}{\bar{a}+\bar{b}}s_{i,3})$.
Thus $b_{i+1}=[\partial f]^{-1}(s_{i}^{\prime})$ would be such that
$b_{i+1}<b_{i}$. Since $\partial f(b_{i+1})=[s_{i+1,3},s_{i+1,4}]$,
we see that $s_{i+1,3}\leq\frac{\bar{b}}{\bar{a}+\bar{b}}s_{i,3}$,
which implies $\lim_{i\to\infty}s_{i+1,3}=0$. Thus $\lim_{i\to\infty}b_{i}=b^{\prime}$
and we are done.

\textbf{Case B: $a^{\prime}<\bar{a}$ and $b^{\prime}<\bar{b}$.}

If case (1) is not encountered, then $a_{i}>a^{\prime}$ and $b_{i}>b^{\prime}$
for all $i$. We prove that case (2) must hold.

Consider $s_{i}^{\prime}=\frac{1}{a_{i}+b_{i}}\int_{-a_{i}}^{b_{i}}\partial f(x)dx$.
Let 
\[
\partial f(-a_{i})=[s_{i,1},s_{i,2}]\mbox{ and }\partial f(b_{i})=[s_{i,3},s_{i,4}].
\]
By the monotonicity of $\partial f(\cdot)$, we have $s_{i,1}\leq s_{i,2}\leq0\leq s_{i,3}\leq s_{i,4}$.
We have $s_{i}^{\prime}\in[s_{i,2},s_{i,3}]$, and $s_{i}^{\prime}$
equals $s_{i,2}$ only if $\partial f(x)=\{s_{i,2}\}$ for all $x\in(-a_{i},b_{i})$.
This cannot happen as $a_{i}>a^{\prime}$ would ensure that $s_{i,2}<0$,
and $b_{i}>b^{\prime}$ would then imply $0\notin\partial f(0)$,
which is a contradiction. We can also argue that $s_{i}^{\prime}$
cannot be $s_{i,3}$. Thus $s_{i}^{\prime}\in(s_{i,2},s_{i,3})$.
By the workings of Algorithm \ref{alg:fn-bracketing}, we either have
$-a_{i+1}\in[\partial f]^{-1}(s_{i}^{\prime})$ or $b_{i+1}\in[\partial f]^{-1}(s_{i}^{\prime})$,
which will mean that either $a_{i+1}<a_{i}$ or $b_{i+1}<b_{i}$.
Thus the sequences $\{a_{i}\}_{i}$ and $\{b_{i}\}_{i}$ are nonincreasing,
and for each $i$, either $a_{i+1}<a_{i}$ or $b_{i+1}<b_{i}$. 

Let $b^{*}:=\lim_{i\to\infty}b_{i}$ and $a^{*}:=\lim_{i\to\infty}a_{i}$.
It is clear that $b^{*}\geq b^{\prime}$ and $a^{*}\geq a^{\prime}$.
We prove that $b^{\prime}=b^{*}$ and $a^{\prime}=a^{*}$. Let $\partial f(a^{*})=[s_{1}^{*},s_{2}^{*}]$
and $\partial f(b^{*})=[s_{3}^{*},s_{4}^{*}]$. It is clear that $s_{1}^{*}\leq s_{2}^{*}\leq0\leq s_{3}^{*}\leq s_{4}^{*}$.
If $a^{*}>a^{\prime}$, then $s_{2}^{*}<0$. Otherwise $b^{*}>b^{\prime}$
gives $s_{3}^{*}>0$. In either case, we have $s_{1}^{*}\leq s_{2}^{*}<s_{3}^{*}\leq s_{4}^{*}$.
Now, 
\[
\lim_{i\to\infty}\frac{1}{a_{i}+b_{i}}\int_{-a_{i}}^{b_{i}}\partial f(x)dx=\frac{1}{a^{*}+b^{*}}\int_{-a^{*}}^{b^{*}}\partial f(x)dx.
\]
Since either $a^{*}>a^{\prime}\geq0$ or $b^{*}>b^{\prime}\geq0$,
we have $a^{*}+b^{*}>0$, so the limit above is well defined. Let
this limit be $s^{*}$. It is clear that $s^{*}\in[s_{2}^{*},s_{3}^{*}]$. 

\textbf{Claim: If $s^{*}\in\{s_{2}^{*},s_{3}^{*}\}$, then $a^{*}=a^{\prime}$
and $b^{*}=b^{\prime}$.}

Consider the case when $s^{*}=s_{2}^{*}$. We must have 
\begin{equation}
\partial f(x)=\{s_{2}^{*}\}\mbox{ for all }x\in(-a^{*},b^{*}).\label{eq:partial-f-x-s-2-1}
\end{equation}
If $b^{*}>0$, then the inequality $s_{2}^{*}\leq0$ and $\partial f(x)\in[0,\infty)$
for all $x\in(0,b^{*})$ forces $s_{2}^{*}=0$, which gives $b^{*}\leq b^{\prime}$,
and in turn $b^{*}=b^{\prime}$. We are left with showing that $a^{*}=a^{\prime}$. 

Seeking a contradiction, suppose $a^{*}=a^{\prime}$. Recall that
this implies $s_{2}^{*}<0$. Since $\partial f(x)\in[0,\infty)$ for
all $x\in(0,b^{*})$, \eqref{eq:partial-f-x-s-2-1} implies $b^{*}=0$.
So $\partial f(x)=\{s_{2}^{*}\}$ for all $x\in(-a^{*},0)$. Let $\gamma>0$
be such that 
\[
\int_{-a^{*}}^{\gamma}\partial f(x)dx=a^{*}s_{2}^{*}+\int_{0}^{\gamma}\partial f(x)dx<0.
\]
The local Lipschitz continuity of $f(\cdot)$ at $0$ implies that
$\partial f(\cdot)$ is locally bounded at $0$, so such a $\gamma$
must exist. If $b_{i}<\gamma$, then $\int_{-a_{i}}^{b_{i}}\partial f(x)dx\leq\int_{-a^{*}}^{\gamma}\partial f(x)dx<0$,
so $[\partial f]^{-1}(\frac{1}{a_{i}+b_{i}}\int_{-a_{i}}^{b_{i}}\partial f(x)dx)<0$.
This means that only $a_{i}$ would decrease and $b_{i}$ would remain
constant, contradicting the fact that $b^{*}=\lim_{i\to\infty}b_{i}=0$.
This ends the proof of our claim when $s^{*}=s_{2}^{*}$. The case
when $s^{*}=s_{3}^{*}$ is similar. This ends the proof of our claim.

Recalling the situation before our claim, we have $s^{*}\in(s_{2}^{*},s_{3}^{*})$.
This means that either $-a_{i+1}\in(-a^{*},b^{*})$ or $b_{i+1}\in(-a^{*},b^{*})$
for $i$ large enough, which contradicts the definition of $a^{*}$
and $b^{*}$. So $b^{\prime}=b^{*}$ and $a^{\prime}=a^{*}$ as needed.
\end{proof}
Theorem \ref{thm:interval-analysis} shows that if $0\notin\partial f(-a_{i})$
and $0\notin\partial f(b_{i})$ for all $i$, the only situation when
the iterates $\{a_{i}\}_{i}$ and $\{b_{i}\}_{i}$ of Algorithm \ref{alg:fn-bracketing}
do not both converge to zero is when both $a^{\prime}=\lim_{i\to\infty}a_{i}$
and $b^{\prime}=\lim_{i\to\infty}b_{i}$ are such that $-a^{\prime}$
and $b^{\prime}$ minimize $f(\cdot)$, and $0\in[-a^{\prime},b^{\prime}]$.
When $0\in\partial f(-a_{i})$ or $0\in\partial f(b_{i})$ for some
iterate $a_{i}>0$ or $b_{i}>0$, we can assume without loss of generality
that $0\in\partial f(-a_{i})$. Algorithm \ref{alg:fn-bracketing}
would continue with the iterates $a_{i}$ staying put, and $\{b_{i}\}_{i}$
strictly decreasing to a minimizer of $f(\cdot)$. The point $0$
would lie in $[-a^{\prime},b^{\prime}]$.

The observation in the last paragraph shows the following behavior
of Algorithm \ref{alg:simple-enh-vN}: When there is a nontrivial
line segment on $\partial S$ such that $0$ lies somewhere on the
line segment, the cluster points of the iterates $\{s_{i}\}_{i}$
of Algorithm \ref{alg:simple-enh-vN} will land on the line segment.
Furthermore, $0$ lies in the convex hull of the cluster points of
$\{s_{i}\}_{i}$.

There is no fixed behavior of the iterates $\{a_{i}\}_{i}$ and $\{b_{i}\}_{i}$
of Algorithm \ref{alg:fn-bracketing}, as the following result shows.
\begin{prop}
\label{prop:any-sequence}(Arbitrary decrease in width) Let $\bar{a}$
and $\bar{b}$ be two positive numbers, and $f:[-\bar{a},\bar{b}]\to\mathbb{R}$.
Let the nonincreasing, nonnegative sequences $\{a_{i}\}_{i}$ and
$\{b_{i}\}_{i}$ be such that 
\begin{enumerate}
\item $a_{0}=\bar{a}$ and $b_{0}=\bar{b}$, and 
\item For each $i$, either $a_{i+1}<a_{i}$ and $b_{i+1}=b_{i}$, or $a_{i+1}=a_{i}$
and $b_{i+1}<b_{i}$. 
\end{enumerate}
We can choose a proper convex function $f:[-\bar{a},\bar{b}]\to\mathbb{R}$
such that Algorithm \ref{alg:fn-bracketing} generates the iterates
$\{a_{i}\}_{i}$ and $\{b_{i}\}_{i}$. \end{prop}
\begin{proof}
Let $a^{\prime}:=\lim_{i\to\infty}a_{i}$ and $b^{\prime}:=\lim_{i\to\infty}b_{i}$.
Define $f(\cdot)$ to be zero on $[-a^{\prime},b^{\prime}]$. We now
define $f(\cdot)$ on the rest of $[-\bar{a},\bar{b}]$. Construct
the sequences of nonincreasing positive numbers $\{\alpha_{i}\}_{i}$,
$\{\beta_{i}\}_{i}$ and $\{\gamma_{i}\}_{i}$ satisfying the following
rules:
\begin{enumerate}
\item [(A)]If $a_{i+1}<a_{i}$ and $b_{i+1}=b_{i}$, then \textrm{$[a_{i+1}+b_{i}]\gamma_{i}+[a_{i}-a_{i+1}][-\alpha_{i}]<[a_{i}+b_{i}][-\gamma_{i}]$}
\item [(B)]If $a_{i+1}=a_{i}$ and $b_{i+1}<b_{i}$, then \textrm{$[a_{i}+b_{i+1}][-\gamma_{i}]+[b_{i}-b_{i+1}][\beta_{i}]>[a_{i}+b_{i}][\gamma_{i}]$}
\item [(C)]$\alpha_{i+1}=\beta_{i+1}=\gamma_{i}$ and $\gamma_{i+1}\leq\gamma_{i}$
for all $i\geq0$.
\item [(D)]$\alpha_{0}=\beta_{0}=1$.
\end{enumerate}
We can construct the sequences inductively with $\alpha_{0}$ and
$\beta_{0}$ defined through (D), $\gamma_{i}$ defined by $\alpha_{i}$
and $\beta_{i}$ through (A) and (B), and $\alpha_{i+1}$ and $\beta_{i+1}$
defined by $\gamma_{i}$ through (C). Define $\partial f(\cdot)$
by 
\[
\partial f(x):=\begin{cases}
\{-\alpha_{i}\} & \mbox{ if }x\in(-a_{i},-a_{i+1})\\
\{\beta_{i}\} & \mbox{ if }x\in(b_{i+1},b_{i}).
\end{cases}
\]
The function $f(\cdot)$ can be inferred from $f(x)=\int_{0}^{x}\partial f(x)dx$
since the monotone function $\partial f(\cdot)$ is integrable. We
now verify that Algorithm \ref{alg:fn-bracketing} applied to $f(\cdot)$
generates the sequence $\{a_{i}\}_{i}$ and $\{b_{i}\}_{i}$. We first
look at the case where $a_{i+1}<a_{i}$ and $b_{i+1}=b_{i}$. Here,
\begin{eqnarray*}
\int_{-a_{i}}^{b_{i}}\partial f(x)dx & = & \int_{-a_{i+1}}^{b_{i}}\partial f(x)dx+\int_{-a_{i}}^{-a_{i+1}}\partial f(x)dx\\
 & \leq & [a_{i+1}+b_{i}]\gamma_{i}+[a_{i}-a_{i+1}][-\alpha_{i}]\\
 & < & [a_{i}+b_{i}][-\gamma_{i}]\\
\Rightarrow\frac{1}{a_{i}+b_{i}}\int_{-a_{i}}^{b_{i}}\partial f(x)dx & < & -\gamma_{i}.
\end{eqnarray*}
It is clear that $\frac{1}{a_{i}+b_{i}}\int_{-a_{i}}^{b_{i}}\partial f(x)dx>-\alpha_{i}$.
Since condition (C) implies that $\partial f(x)\in[-\gamma_{i},\gamma_{i}]$
for all $x\in(-a_{i+1},b_{i+1})$, this implies that $[\partial f]^{-1}(\frac{1}{a_{i}+b_{i}}\int_{-a_{i}}^{b_{i}}\partial f(x)dx)=-a_{i+1}$.
This means that from the end points $-a_{i}$ and $b_{i}$ at iteration
$i$, the next endpoints are indeed $-a_{i+1}$ and $b_{i}$ as claimed.
The case when $a_{i+1}=a_{i}$ and $b_{i+1}<b_{i}$ is similar.
\end{proof}
Even though Proposition \ref{prop:any-sequence} shows that the width
of the intervals can decrease at any rate in Algorithm \ref{alg:fn-bracketing},
the proof of case 2 in Theorem \ref{thm:lin-conv-best} shows that
$\{\|y_{i}\|\}_{i}$ converges quickly. 

We give conditions such that the width of the intervals in Algorithm
\ref{alg:fn-bracketing} decreases at a linear rate.
\begin{thm}
\label{thm:lin-conv-1}(Bracketing in Algorithm \ref{alg:fn-bracketing})
Let $f:\mathbb{R}\to\mathbb{R}$ be a convex function such that $f(0)=0$
and $f(\cdot)$ is differentiable at $0$ with $f^{\prime}(0)=0$
in Algorithm \ref{alg:fn-bracketing}. Suppose further that $\partial f(\cdot)$
has left derivative $f_{-}^{\prime\prime}(0)$ and right derivative
$f_{+}^{\prime\prime}(0)$ which are formally defined as 
\begin{equation}
f_{-}^{\prime\prime}(0):=\lim_{t\searrow0,v\in\partial f(-t)}\frac{1}{-t}[v-f^{\prime}(0)]\mbox{ and }f_{+}^{\prime\prime}(0):=\lim_{t\searrow0,v\in\partial f(t)}\frac{1}{t}[v-f^{\prime}(0)].\label{eq:LR-derv-subdiff}
\end{equation}
 In view of the convexity of $f$ (i.e. monotonicity of $\partial f(\cdot)$),
we have $f_{-}^{\prime\prime}(0)\geq0$ and $f_{+}^{\prime\prime}(0)\geq0$.
Suppose $f_{-}^{\prime\prime}(0)>0$ and $f_{+}^{\prime\prime}(0)>0$.
Then the width of the interval $[-a_{i},b_{i}]$, easily seen to be
$a_{i}+b_{i}$, decreases at a linear rate.
\end{thm}
The formulas $f_{-}^{\prime\prime}(0)$ and $f_{+}^{\prime\prime}(0)$
are defined by \eqref{eq:LR-derv-subdiff} instead of 
\[
f_{-}^{\prime\prime}(0)=\lim_{t\searrow0}\frac{1}{-t}[f^{\prime}(-t)-f^{\prime}(0)]\mbox{ and }f_{+}^{\prime\prime}(0):=\lim_{t\searrow0}\frac{1}{t}[f^{\prime}(t)-f^{\prime}(0)],
\]
 because \eqref{eq:LR-derv-subdiff} does not require the differentiability
of $f(\cdot)$ in a neighborhood of $0$ and is more general. We now
prove Theorem \ref{thm:lin-conv-1}.
\begin{proof}
In view of the existence of the limits in \eqref{eq:LR-derv-subdiff},
for any constants $g_{l,l}$ and $g_{l,u}$ such that $0<g_{l,l}<f_{-}^{\prime\prime}(0)<g_{l,u}$,
we can find $\delta>0$ such that if $x\in[-\delta,0)$, then $\partial f(x)\subset[-|x|g_{l,u},-|x|g_{l,l}]$.
Similarly, for any constants $g_{r,l}$ and $g_{r,u}$ such that $0<g_{r,l}<f_{+}^{\prime\prime}(0)<g_{r,u}$,
we can reduce $\delta>0$ if necessary so that if $x\in(0,\delta]$,
then $\partial f(x)\subset[xg_{r,l},xg_{r,u}]$. We shall also assume
that 
\begin{equation}
\frac{g_{l,u}}{g_{l,l}}<R\mbox{ and }\frac{g_{r,u}}{g_{r,l}}<R\mbox{, where }R\mbox{ can be made arbitrarily close to }1.\label{eq:ratio-R-formula}
\end{equation}

After one iteration, the interval $[-a_{i},b_{i}]$ becomes either
$[c_{i},b_{i}]$ or $[-a_{i},c_{i}]$, depending on the sign of $c_{i}$.
We now try to find upper and lower bounds on $c_{i}$. Suppose $a_{i}\in(0,\delta]$
and $b_{i}\in(0,\delta]$. Then the graph of $\partial f(\cdot)$
on $[-\delta,\delta]$ is bounded from below by the piecewise linear
function $h_{L}:[-\delta,\delta]\to\mathbb{R}$ and the from above
by the piecewise linear function $h_{U}:[-\delta,\delta]\to\mathbb{R}$
(See Figure \ref{fig:wedge}) defined respectively by 
\[
h_{L}(x):=\begin{cases}
g_{l,u}x & \mbox{ if }x\in[-\delta,0]\\
g_{r,l}x & \mbox{ if }x\in[0,\delta]
\end{cases}\mbox{ and }h_{U}(x):=\begin{cases}
g_{l,l}x & \mbox{ if }x\in[-\delta,0]\\
g_{r,u}x & \mbox{ if }x\in[0,\delta].
\end{cases}
\]

\begin{figure}[!h]
\scalebox{0.5}[0.5]{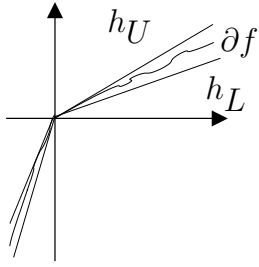}\caption{\label{fig:wedge}Illustration of how $h_{L}(\cdot)$ and $h_{U}(\cdot)$
compare to $\partial f(\cdot)$. }

\end{figure}

We now estimate an upper bound on $c$. An upper bound on $\frac{1}{a_{i}+b_{i}}[f(b_{i})-f(-a_{i})]$,
which equals $\frac{1}{a_{i}+b_{i}}\int_{-a_{i}}^{b_{i}}f^{\prime}(x)dx$
since $f^{\prime}(\cdot)$ is integrable, is $v_{U}:=\frac{1}{a_{i}+b_{i}}\int_{-a_{i}}^{b_{i}}h_{U}(x)dx$.
Since $h_{L}(\cdot)\leq\partial f(\cdot)$, $c_{i}\leq h_{L}^{-1}(v_{U})$.
We now proceed to calculate these values.

We can calculate that $v_{U}:=\frac{1}{a_{i}+b_{i}}[-\frac{1}{2}g_{l,l}a_{i}^{2}+\frac{1}{2}g_{r,u}b_{i}^{2}]$.
We are interested in the upper bound of $c_{i}$ in the case when
the next interval is $[-a_{i},c_{i}]$, so we only consider the case
where $v_{U}>0$. In this case, $h_{L}^{-1}(v_{U})=\frac{1}{g_{r,l}[a_{i}+b_{i}]}[-\frac{1}{2}g_{l,l}a_{i}^{2}+\frac{1}{2}g_{r,u}b_{i}^{2}]$.
The width of the interval $[-a_{i},c_{i}]$ divided by the width of
$[a_{i},b_{i}]$ is estimated as follows. 
\begin{eqnarray*}
\frac{a_{i}+c_{i}}{a_{i}+b_{i}} & \leq & \frac{a_{i}+h_{L}^{-1}(v_{U})}{a_{i}+b_{i}}\\
 & = & \frac{2g_{r,l}[a_{i}+b_{i}]a_{i}+[-g_{l,l}a_{i}^{2}+g_{r,u}b_{i}^{2}]}{2g_{r,l}[a_{i}+b_{i}]^{2}}\\
 & = & \frac{1}{2}+\frac{[g_{r,l}-g_{l,l}]a_{i}^{2}+[g_{r,u}-g_{r,l}]b_{i}^{2}}{2g_{r,l}[a_{i}+b_{i}]^{2}}\\
 & \leq & \frac{1}{2}+\frac{[1-\frac{g_{l,l}}{g_{r,l}}]a_{i}^{2}+[R-1]b_{i}^{2}}{2[a_{i}+b_{i}]^{2}}\\
 & \leq & \frac{1}{2}+\frac{1}{2}\underbrace{\left(\left[1-\frac{g_{l,l}}{g_{r,l}}\right]+[R-1]\right)}_{(*)},
\end{eqnarray*}
where the ratio $R$ is as defined in \eqref{eq:ratio-R-formula}.
The term $[R-1]$ can be made arbitrarily close to zero. The term
$\frac{g_{l,l}}{g_{r,l}}$ can be made arbitrarily close to $t:=f_{-}^{\prime\prime}(0)/f_{+}^{\prime\prime}(0)$.
In other words, $[1-\frac{g_{l,l}}{g_{r,l}}]$ can be arbitrarily
close to $[1-t]$. If $[1-t]<0$, then with proper choices of $g_{l,l}$,
$g_{r,l}$, $g_{l,u}$ and $g_{r,u}$, we can make $(*)$ negative,
in which case the ratio $\frac{a_{i}+c_{i}}{a_{i}+b_{i}}$ is less
than $1/2$. If $[1-t]\geq0$, we still have $[1-t]<1$, so with the
proper choice of constants, we can ensure that $\frac{a_{i}+c_{i}}{a_{i}+b_{i}}\leq\frac{3}{4}+\frac{1}{4}[1-t]$,
which still ensures that the reduction of the width of the intervals
is still linear.

The calculations for finding a lower bound on $c_{i}$ is similar.
The lower bound is of interest when the next interval is $[c_{i},b_{i}]$,
and that $c_{i}<0$. Thus $c_{i}<h_{U}^{-1}(v_{L})$, where $v_{L}:=\frac{1}{a_{i}+b_{i}}[-\frac{1}{2}g_{l,u}a_{i}^{2}+\frac{1}{2}g_{r,l}b_{i}^{2}]$.
So  
\begin{eqnarray*}
\frac{-c_{i}+b_{i}}{a_{i}+b_{i}} & \leq & \frac{-h_{U}^{-1}(v_{L})+b_{i}}{a_{i}+b_{i}}\\
 & = & \frac{1}{2}+\frac{[g_{l,u}-g_{l,l}]a_{i}^{2}+[g_{l,l}-g_{r,l}]b_{i}^{2}}{2g_{l,l}[a_{i}+b_{i}]^{2}}\\
 & \leq & \frac{1}{2}+\frac{[R-1]a_{i}^{2}+[1-\frac{g_{r,l}}{g_{l,l}}]b_{i}^{2}}{2[a_{i}+b_{i}]^{2}}\\
 & \leq & \frac{1}{2}+\frac{1}{2}\left([R-1]+\left[1-\frac{g_{r,l}}{g_{l,l}}\right]\right).
\end{eqnarray*}
Once again, the ratio is $\frac{g_{r,l}}{g_{l,l}}$ can be chosen
arbitrarily close to $1/t$, where $t$ was as defined earlier. If
$[1-\frac{1}{t}]<0$, we will have $\frac{-c_{i}+b_{i}}{a_{i}+b_{i}}<\frac{1}{2}$
eventually. If $[1-\frac{1}{t}]>0$, we still have $[1-\frac{1}{t}]<1$,
in which case we can ensure that $\frac{a_{i}+c_{i}}{a_{i}+b_{i}}\leq\frac{3}{4}+\frac{1}{4}[1-\frac{1}{t}]$.
No matter the case, we have a linear rate of convergence of the width
of the intervals to zero.\end{proof}
\begin{cor}
\label{cor:lin-conv-2}(Linear convergence of Algorithm \ref{alg:fn-bracketing})
With the additional assumptions in Theorem \ref{thm:lin-conv-1},
the iterates of Algorithm \ref{alg:fn-bracketing} are such that the
sequence 
\begin{equation}
\big\{ d\big((0,0),\{(-a_{i},f(-a_{i})),(b_{i},f(b_{i}))\}\big)\big\}_{i},\label{eq:min-formula}
\end{equation}
where the distance in $\mathbb{R}^{2}$ is measured by the 2-norm,
is bounded above by a linearly convergence sequence. The corresponding
sequence $\{\|y_{i}\|\}_{i}$ in Algorithm \ref{alg:simple-enh-vN}
is bounded by a linearly convergent sequence. \end{cor}
\begin{proof}
By Theorem \ref{thm:lin-conv-1}, the width of the intervals $[a_{i},b_{i}]$
converges linearly to zero. Hence $\{\min(a_{i},b_{i})\}_{i}$ is
bounded by a linearly convergent sequence. Note that $f(\cdot)$,
being convex, is locally Lipschitz at $0$ with some constant $L$,
so $f(-a_{i})\leq La_{i}$ and $f(b_{i})\leq Lb_{i}$. We can thus
easily obtain the first conclusion.

The points $(0,0)$, $(-a_{i},f(-a_{i}))$ and $(b_{i},f(b_{i}))$
are points in the epigraph of $f(\cdot)$, and the formula in \eqref{eq:min-formula}
is an upper bound on the distance from $(0,0)$ to the line segment
connecting the points $(-a_{i},f(-a_{i}))$ and $(b_{i},f(b_{i}))$.
Hence the second statement is clear. 
\end{proof}
The assumptions of Theorem \ref{thm:lin-conv-1} correspond to a second
order property on the boundary of $S$ at $0$. With added structure,
Algorithm \ref{alg:fn-bracketing} and \ref{alg:simple-enh-vN} can
converge faster. For example, if $S$ is polyhedral and Algorithm
\ref{alg:simple-enh-vN} chooses the extreme points, we have finite
convergence of $y_{i}$ to $0$ because there are only finitely many
extreme points for a polyhedron. 
\begin{rem}
\label{rem:difficulties}(Difficulties in extending to $m>2$) For
much of this section and the last, we analyzed the case where $m=2$
in Algorithm \ref{alg:enhanced-vN}. We expect the extension to $m>2$
to be difficult, and the following are some of the reasons.
\begin{enumerate}
\item We made a connection to monotonicity of $\partial f(\cdot)$ here
and proved our results using single variable analysis. These need
to be extended to higher dimensions for $m>2$.
\item Proposition \ref{thm:bisec-enhanced-vN} cannot be easily extended
to the higher dimensional case. It is not necessarily true that for
the higher dimensional case, the projection will be on a face that
is of codimension 1.
\item For the 2 dimensional case, we see that $\aff(\{\bar{c}_{i},s_{i}\})\cap S$
is equal to $[\bar{c}_{i},s_{i}]$ if $0\notin[\bar{c}_{i},s_{i}]$.
One can see that if $S$ is a sphere in $\mathbb{R}^{3}$, for any
3 points $a$, $b$ and $c$ on $\partial S$, we do not have $\aff(\{a,b,c\})\cap S=\conv(\{a,b,c\})$. 
\item The projection onto the convex hull of two points is easy, and we
can write down an analytic formula to help in our analysis. However,
it is difficult to write down such a formula for the projection onto
the convex hull of 3 or more points in higher dimensions, even if
this projection can be solved quite effectively using the methods
discussed earlier. 
\end{enumerate}
\end{rem}

\section{Numerical experiments}

We perform some numerical experiments to show that Algorithm \ref{alg:enhanced-vN}
is more effective for some problem instances. 

We generate our random matrices $A\in\mathbb{R}^{30\times80000}$
using the following code segment in Matlab:\texttt{
\begin{eqnarray}
 &  & \mbox{A=rand(30,80000)-ones(30,80000)*0.315;}\label{eq:generate-A}\\
 &  & \mbox{for i=1:80000}\nonumber \\
 &  & \quad\mbox{A(:,i)=A(:,i)/norm(A(:,i));}\nonumber \\
 &  & \mbox{end}\nonumber 
\end{eqnarray}
}Through our experiments, we found that this choice of parameters
generate problem instances for which either \eqref{eq:von-Neumann-ineq}
is feasible (in which case the von Neumann algorithm cannot converge
finitely), or \eqref{eq:perceptron-ineq} is feasible but the von
Neumann algorithm typically takes many iterations, sometimes more
than 2000 iterations, before it terminates.

\subsection{Numerical experiment 1: Comparison against von Neumann algorithm
when $A^{T}y>0$ feasible}

We ran experiments for 491 different matrices $A\in\mathbb{R}^{30\times80000}$
generated by \eqref{eq:generate-A} such that \eqref{eq:perceptron-ineq}
holds (i.e., $0$ does not lie in the convex hull of the elements
generated by the columns of $A$). We calculated the number of iterations
needed for the von Neumann algorithm to find a $y$ satisfying \eqref{eq:perceptron-ineq},
and for Algorithm \ref{alg:enhanced-vN} with various limits on the
size of the active set (See Subsection \ref{sub:agg_strat}) to do
the same. The aggregation strategy is the one in Remark \ref{rem:agg_strat_2},
where we aggregate the oldest element(s) that have not been aggregated.
We set a limit of 2000 for the number of iterations.

We first look at the results obtained from the conducting experiments
on 491 different matrices $A\in\mathbb{R}^{30\times80000}$. We look
at Table \ref{tab:Alg-QP-vs-vNm} for a comparison of the number of
iterations needed by Algorithm \ref{alg:enhanced-vN} to find a $y$
such that $A^{T}y>0$ versus the number of iterations needed by the
von Neumann algorithm. We shall use the following convention in our
diagrams and tables in this section:
\begin{defn}
\label{def:A-N-notation-numerical}Let $A_{N}$ denote Algorithm \ref{alg:enhanced-vN}
where the size of the set $C_{i}$ is bounded above by $N-1$ after
line 12 is performed. In other words, we aggregate according to Remark
\ref{rem:agg_strat_2} when the size of the set $C_{i}$ equals $N$. 
\end{defn}
\begin{table}[!h]
\begin{tabular}{|c|r|r|r|r|r|r|r|r|}
\hline 
\multicolumn{9}{|c|}{Comparing iteration counts of Algorithm \ref{alg:enhanced-vN} against
von Neumann Algorithm}\tabularnewline
\hline 
$N$ & \multicolumn{2}{r|}{$A_{2}>A_{N}$} & \multicolumn{2}{r|}{$A_{2}<A_{N}$} & \multicolumn{2}{r|}{$A_{2}=A_{N}\leq2000$} & \multicolumn{2}{r|}{$A_{2}>2000$ and $A_{N}>2000$ }\tabularnewline
\hline 
\hline 
5 & 63 & 12.9\% & 260 & 53.0\% & 0 & 0\% & 168 & 34.2\%\tabularnewline
\hline 
10 & 170 & 34.6\% & 152 & 31.0\% & 4 & 0.8\% & 165 & 33.6\%\tabularnewline
\hline 
15 & 342 & 69.7\% & 18 & 3.7\% & 2 & 0.4\% & 129 & 26.3\%\tabularnewline
\hline 
20 & 409 & 83.3\% & 1 & 0.2\% & 0 & 0\% & 81 & 16.5\%\tabularnewline
\hline 
25 & 453 & 92.3\% & 0 & 0\% & 0 & 0\% & 38 & 7.7\%\tabularnewline
\hline 
31 & 491 & 100\% & 0 & 0\% & 0 & 0\% & 0 & 0\%\tabularnewline
\hline 
\end{tabular}\caption{\label{tab:Alg-QP-vs-vNm}Refer to the definition of $A_{N}$ in Definition
\ref{def:A-N-notation-numerical}. This table compares the number
of times in 491 experiments where the number of iterations to find
a $y$ s.t. $A^{T}y>0$ for the von Neumann algorithm (or $A_{2}$)
uses is greater than/ less than/ equal to that of $A_{N}$. The last
column represents the number of times both $A_{2}$ and $A_{N}$ reach
their limit of 2001 iterations.}
\end{table}

Recall that $A_{2}$ refers to the von Neumann algorithm (see Remark
\ref{rem:vN-particular-case}). It can be seen that the von Neumann
Algorithm has consistently used fewer iterations than $A_{5}$, and
it is quite competitive with $A_{10}$. As we increase the maximum
size of the active set, the number of iterations needed gets better
compared to the von Neumann algorithm $A_{2}$. 

\begin{figure}[!h]
\includegraphics[scale=0.5]{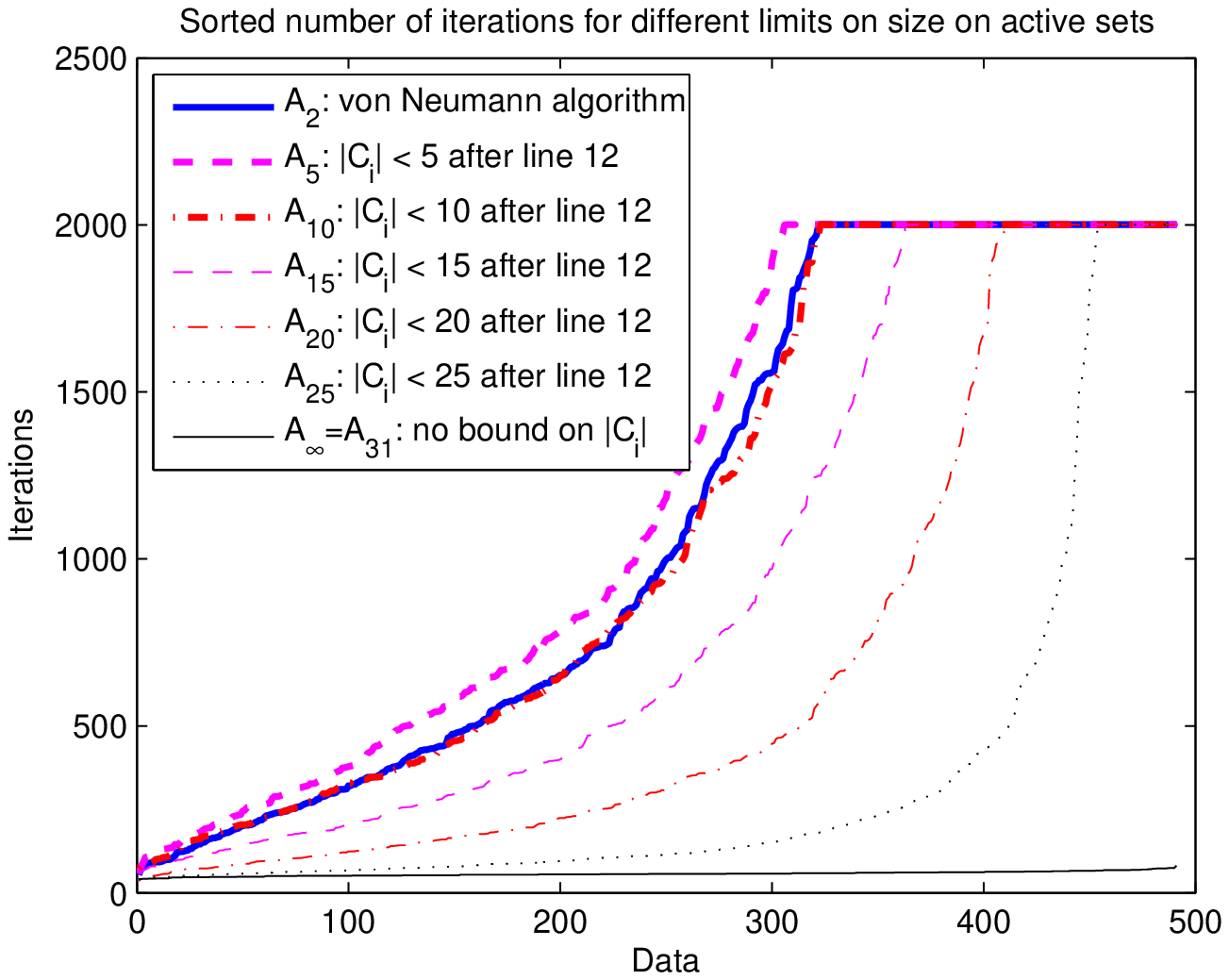}\includegraphics[scale=0.5]{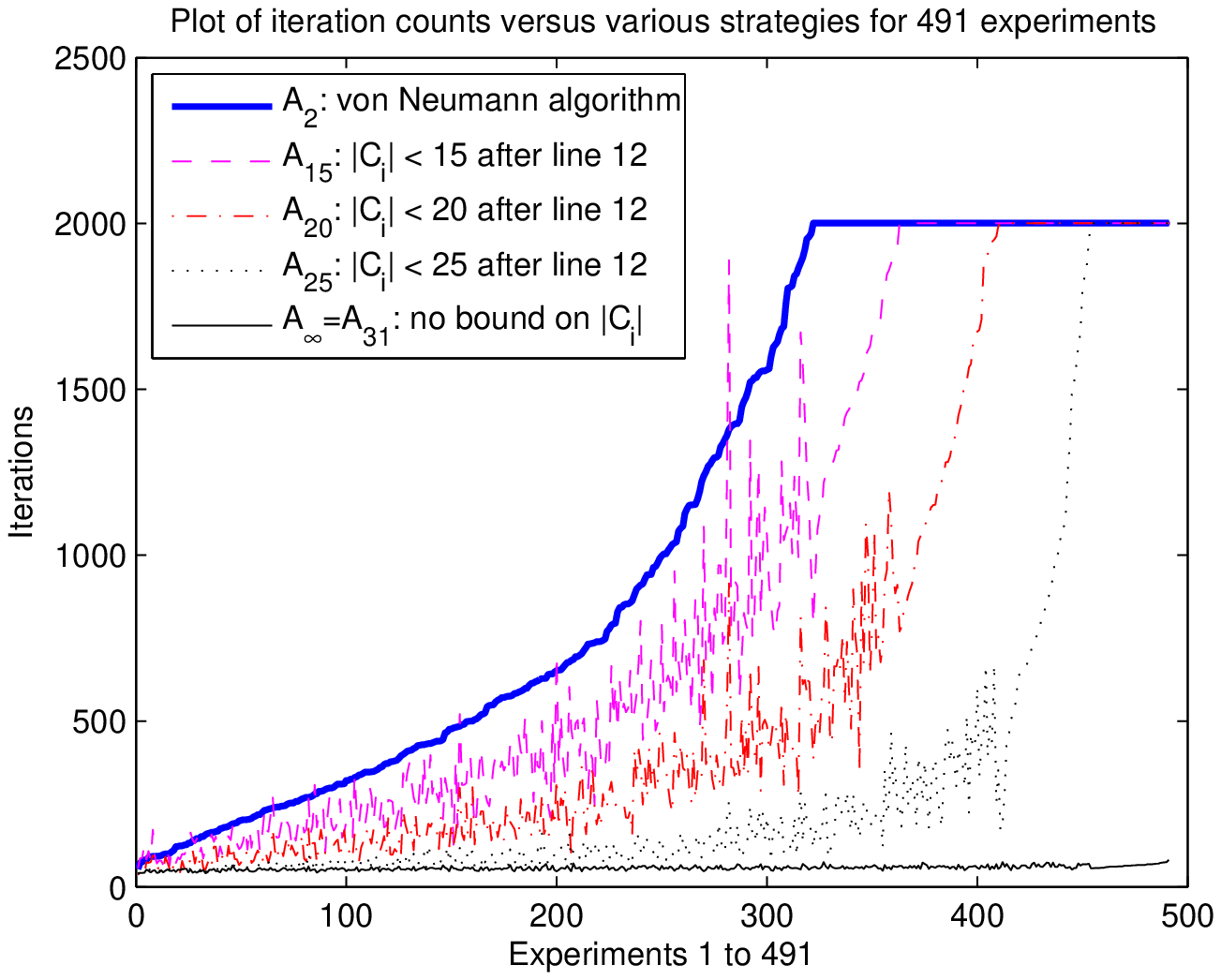}

\caption{\label{fig:exp-histogram}Number of iterations needed for Algorithm
\ref{alg:enhanced-vN} with various parameters to find a $y$ such
that $A^{T}y>0$. In the graph on the right, each vertical line corresponds
to a particular experiment. In the diagram on the left, the number
of iterations needed for all experiments are obtained and sorted.}
\end{figure}

We explain the diagrams in Figure \ref{fig:exp-histogram}, and we
first look at the diagram on the left. In our 491 experiments where
there is a $y$ such that $A^{T}y>0$, we found that overall, the
von Neumann algorithm $A_{2}$ uses fewer iterations to find the $y$
such that $A^{T}y>0$ than $A_{5}$. As we increase the tolerance
of the size of the set $C_{i}$ before we aggregate, the number of
iterations needed to find this $y$ decreases. 

In the diagram on the right of Figure \ref{fig:exp-histogram}, we
sort the experiments so that each vertical line corresponds to a particular
experiment. We see that the von Neumann algorithm usually takes more
iterations than Algorithm \ref{alg:enhanced-vN} than $A_{15}$, though
we notice a few rare instances of when $A_{15}$ takes more iterations
than the von Neumann algorithm. We observe the general pattern that
the larger the tolerance before aggregating $C_{i}$, the fewer iterations
it takes for Algorithm \ref{alg:enhanced-vN}. In fact, when there
is no aggregation, Algorithm \ref{alg:enhanced-vN} takes less than
80 iterations to decide whether \eqref{eq:perceptron-ineq} or \eqref{eq:von-Neumann-ineq}
is feasible.

We now look at a particular anomalous experiment, and explain diagrams
in Figure \ref{fig:exp-data-vNm-QP}. For this particular experiment,
we plot the norm $\|y_{i}\|$ with respect to the iteration $i$.
This example is unusual because the von Neumann algorithm takes fewer
iterations than $A_{20}$. The plots are drawn for each iteration
till we have found $y_{i}$ such that $A^{T}y_{i}>0$. Even though
the von Neumann algorithm takes fewer iterations to get a $y_{i}$
such that $A^{T}y_{i}>0$, the norms of $\|y_{i}\|$ decrease much
slower than all versions of Algorithm \ref{alg:enhanced-vN}. 

\begin{figure}[!h]
\includegraphics[scale=0.5]{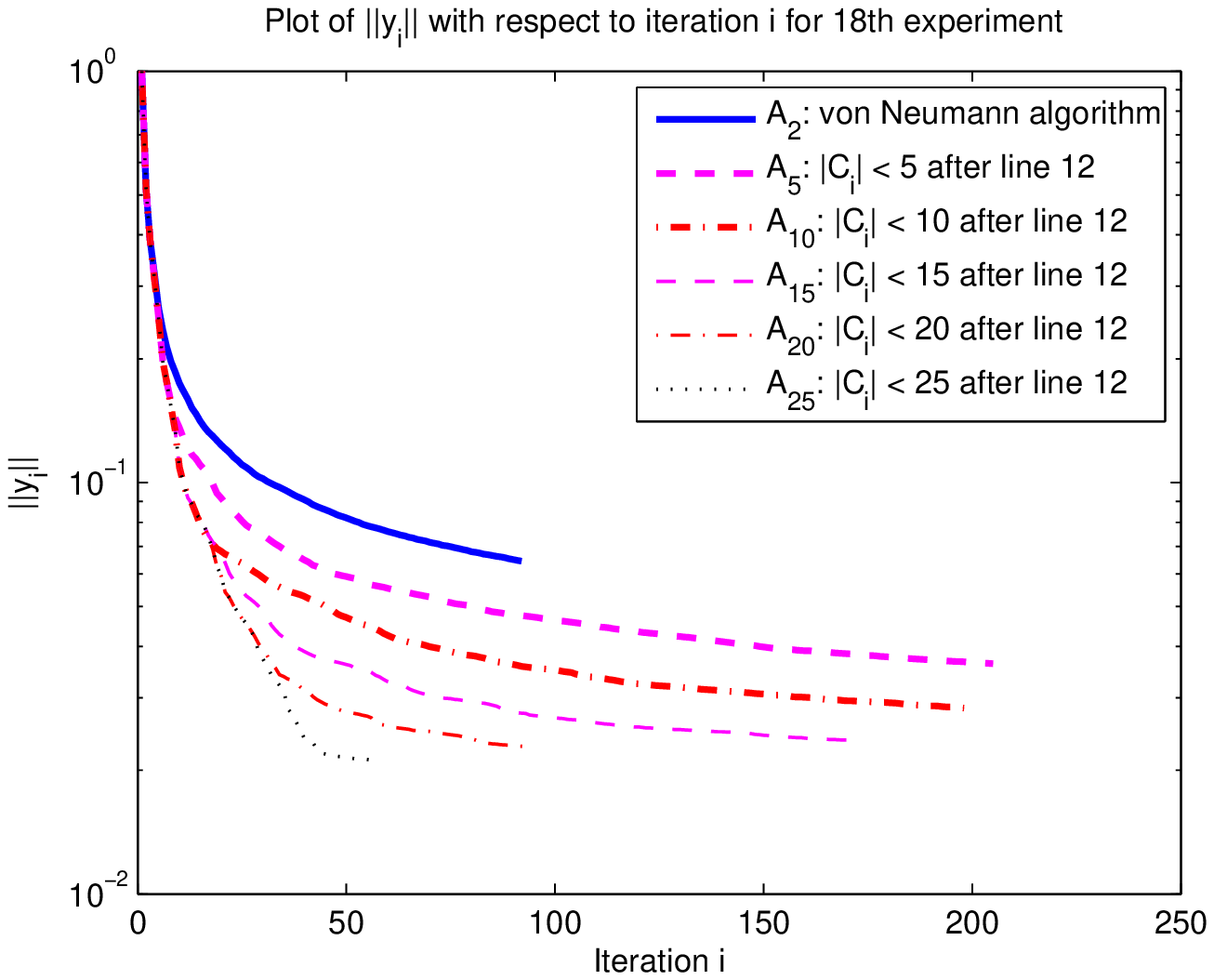}

\includegraphics[scale=0.5]{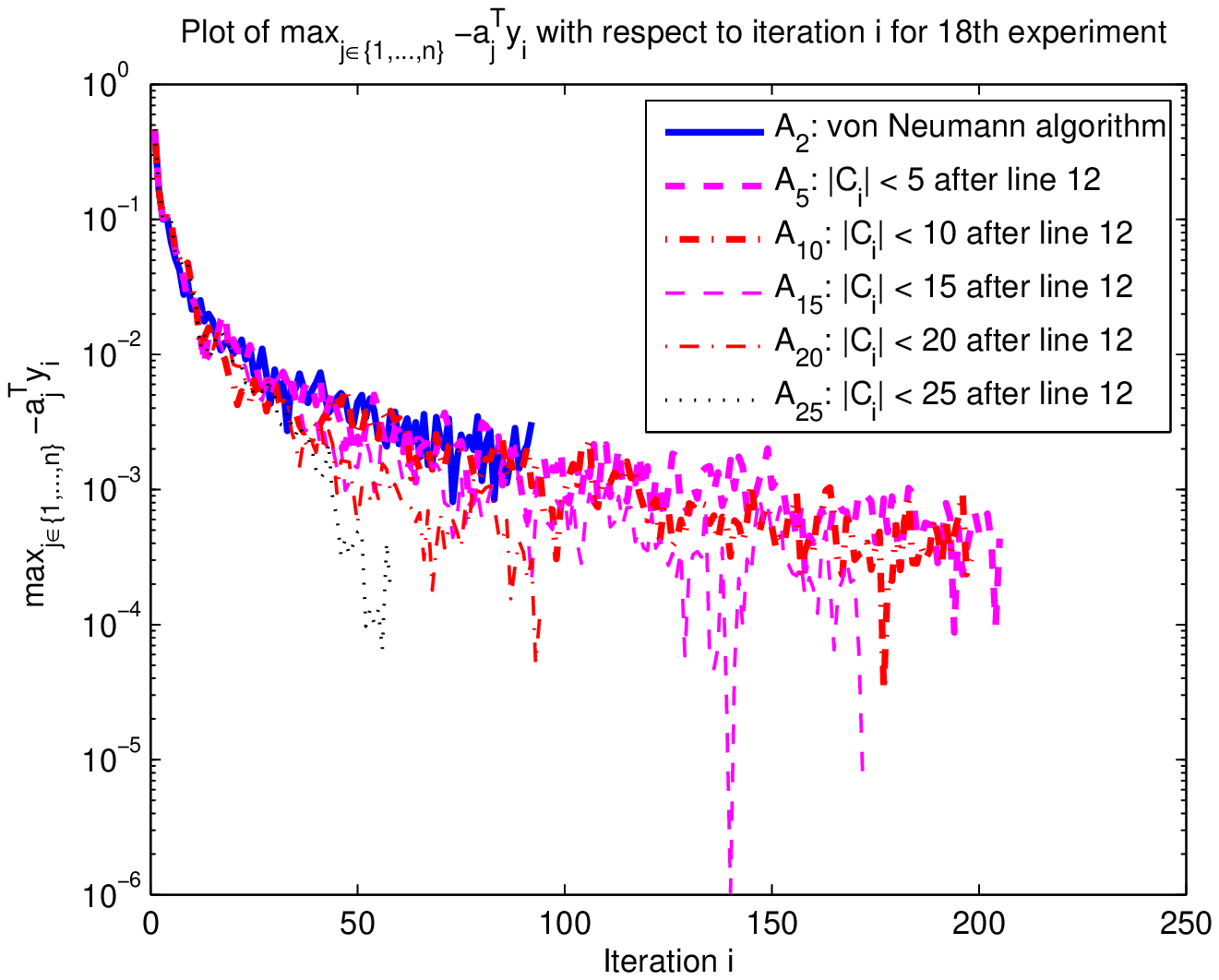}\includegraphics[scale=0.5]{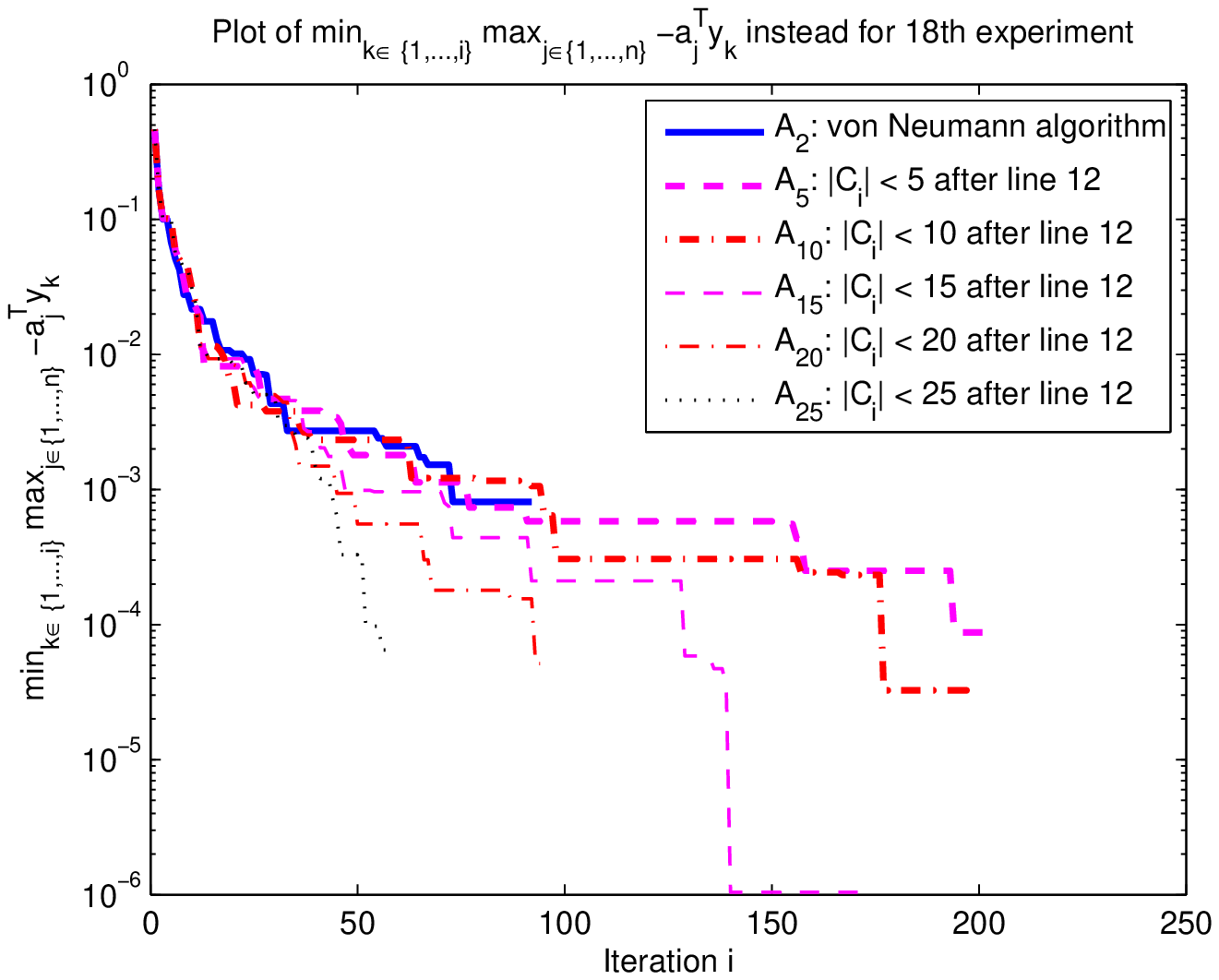}

\caption{\label{fig:exp-data-vNm-QP}The diagrams show an anomalous example
where the von Neumann algorithm takes fewer iterations to find a $y$
such that $A^{T}y>0$ than $A_{20}$.}
\end{figure}

We now explain the bottom diagrams in Figure \ref{fig:exp-data-vNm-QP}.
At each iteration $i$, we calculate 
\[
q_{1,i}:=\max_{j\in\{1,\dots,n\}}-a_{j}^{T}y_{i},
\]
just like in solving \eqref{eq:new-p-i-1-formula}. If this quantity
is negative, then $A^{T}y_{i}>0$, and we end. The bottom left diagram
shows that, other than a general downward trend, there is no clear
pattern in the dependence of $q_{1,i}$ on $i$. In the bottom right
diagram, we calculate 
\[
q_{2,i}:=\min_{k\in\{1,\dots,i\}}[\max_{j\in\{1,\dots,n\}}-a_{j}^{T}y_{i}].
\]
We observe that in general, the larger the limit the size of the active
set, the faster $q_{2,i}$ (and hence $q_{1,i}$) decreases.

\subsection{A note on number of iterations and time}

We have only discussed the performance of the algorithms we test in
terms of iteration counts instead of the time taken. For our experiments
so far, we plot the time taken per iteration versus the number of
iterations for our implementation of Algorithm \ref{alg:enhanced-vN}
as well as our implementation of the von Neumann algorithm in Matlab.
These are shown in Figure \ref{fig:time-per-iter}. The time taken
per iteration for Algorithm \ref{alg:enhanced-vN} is seen to be between
0.0255 seconds to 0.0290 seconds regardless of the size of $C_{i}$.
The time taken per iteration for the von Neumann algorithm is seen
to be between 0.0043 seconds to 0.0047 seconds. Since the running
time of the algorithms and the iteration numbers differ only up to
a constant factor that is implementation dependent, we shall analyze
our algorithms only in terms of the number of iterations. Moreover,
it may be possible to improve this ratio in favor of Algorithm \ref{alg:enhanced-vN}
if the accelerations in Remark \ref{rem:accelerate-primal-QP} are
carried out, especially when $m$ is large. 

\begin{figure}[!h]
\includegraphics[scale=0.5]{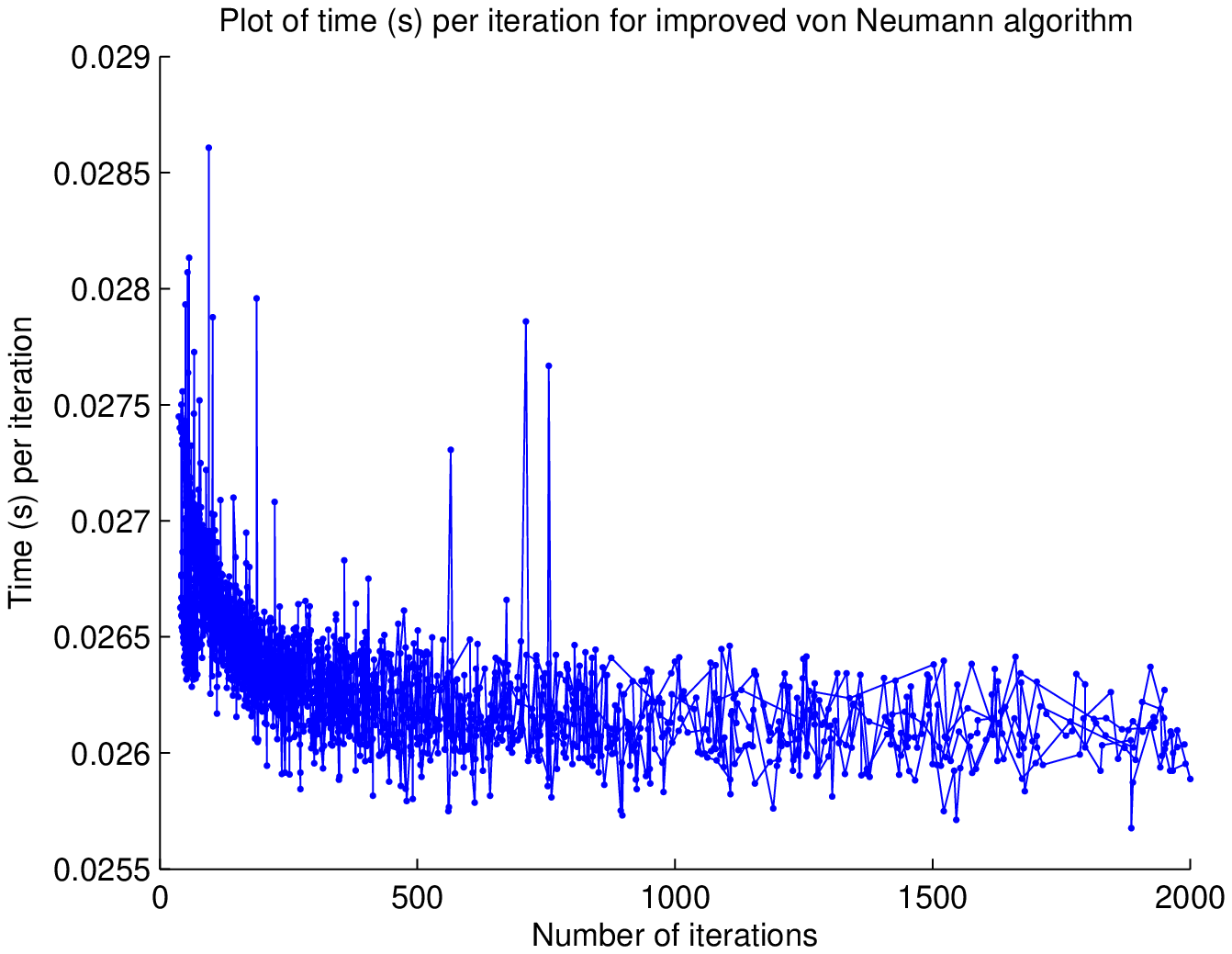}\includegraphics[scale=0.5]{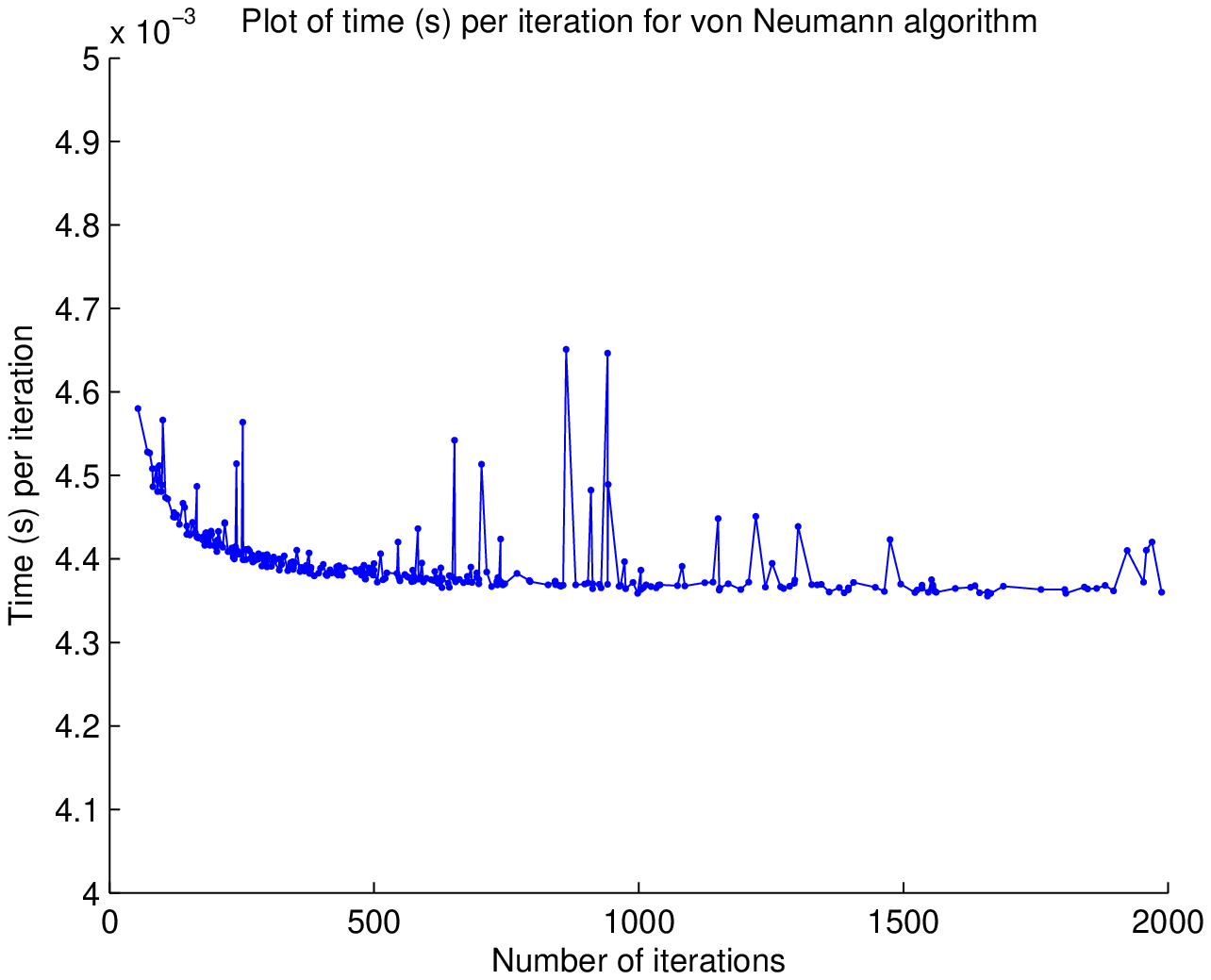}

\caption{\label{fig:time-per-iter}The diagram on the left shows the time taken
per iteration for $A_{5}$, $A_{10}$, $A_{15}$, $A_{20}$, $A_{25}$
and $A_{31}$. It can be seen that the size of the set $C_{i}$ does
not affect the time per iteration. The diagram on the right shows
the time taken per iteration for our implementation of the von Neumann
algorithm. }

\end{figure}

\subsection{\label{sub:agg_strats}Numerical experiment 2: Aggregation strategies}

Table \ref{tab:test_high_agg} below compares the running times of
353 experiments for the case when \eqref{eq:perceptron-ineq} is feasible.
Rows 1-3 look at the number of iterations it takes to find a certificate
vector $y$, while rows 4-6 look at the norms $\|Ax_{i}\|=\|y_{i}\|$
of the iterates at the $80$th iteration for the various aggregation
methods. For the test on the norms $\|Ax_{i}\|$, the undecided column
denotes the number of times at least one algorithm has found a $y$
such that $A^{T}y>0$ in $80$ iterations. The experiments suggest
that the best aggregation method is to aggregate the oldest elements
that have not been aggregated. There might be other factors that we
have not identified which determine the performance of an aggregation
strategy. There could also be better aggregation strategies other
than the ones we have tried.

\begin{table}[h]
\begin{tabular}{|c|c|c|c|c|c|c|}
\hline 
 & 353 runs for \eqref{eq:perceptron-ineq} feasible & \multicolumn{5}{c|}{Best aggregation method}\tabularnewline
\hline 
 &  & (1) & (2) & (3) & Ties & Undecided\tabularnewline
\hline 
\hline 
1 & No. iters for $A_{25}$ to find $y$ solving \eqref{eq:perceptron-ineq} & 178 & 39 & 102 & 34 & 0\tabularnewline
\hline 
2 & No. iters for $A_{20}$ to find $y$ solving \eqref{eq:perceptron-ineq} & 173 & 11 & 110 & 59 & 0\tabularnewline
\hline 
3 & No. iters for $A_{15}$ to find $y$ solving \eqref{eq:perceptron-ineq} & 155 & 10 & 101 & 87 & 0\tabularnewline
\hline 
4 & $\|Ax_{i}\|$ at $i=80$ for $A_{25}$ & 62 & 70 & 90 & 0 & 131\tabularnewline
\hline 
5 & $\|Ax_{i}\|$ at $i=80$ for $A_{20}$ & 142 & 67 & 96 & 0 & 48\tabularnewline
\hline 
6 & $\|Ax_{i}\|$ at $i=80$ for $A_{15}$ & 240 & 37 & 64 & 0 & 12\tabularnewline
\hline 
\hline 
 & 126 runs for \eqref{eq:von-Neumann-ineq} feasible &  &  &  &  & \tabularnewline
\hline 
\hline 
7 & $\|Ax_{i}\|$ at $i=400$ for $A_{15}$ & 114 & 10 & 0 & 0 & 2\tabularnewline
\hline 
8 & $\|Ax_{i}\|$ at $i=400$ for $A_{10}$ & 121 & 5 & 0 & 0 & 0\tabularnewline
\hline 
9 & $\|Ax_{i}\|$ at $i=400$ for $A_{5}$ & 123 & 2 & 1 & 0 & 0\tabularnewline
\hline 
\end{tabular}

\caption{\label{tab:test_high_agg}Experiments on which aggregation strategy
is best among those presented in Remark \ref{rem:agg_strat_2}. The
strategies in the three columns are: (1) oldest aggregated, (2) lowest
coefficient aggregated, and (3) highest coefficient aggregated, as
according to the description in Remark \ref{rem:agg_strat_2}, which
relies on Remark \ref{rem:agg_strat}. For rows 1-6, the 353 experiments
are for when \eqref{eq:perceptron-ineq} is feasible. For rows 7-9,
the 126 experiments are for when \eqref{eq:von-Neumann-ineq} is feasible.
We refer to Subsection \ref{sub:agg_strats} for more details.}
\end{table}

Table \ref{tab:test_high_agg} also compares the running times of
126 experiments for which \eqref{eq:von-Neumann-ineq} is feasible.
When \eqref{eq:von-Neumann-ineq} is feasible, we want to find iterates
$x$ such that $\|Ax\|$ is small. To evaluate the performance of
the aggregation strategies, we look at how the norms of the values
$\|Ax_{i}\|$ vary with the iteration count $i$ for different strategies.
For the test on the norms $\|Ax_{i}\|$, the undecided column denotes
the number of times numerical errors resulting from $\|Ax_{i}\|$
were encountered for at least one algorithm in $400$ iterations.
It is quite clear that the strategy of aggregating the oldest point
obtained is the best strategy among our experiments.

\begin{figure}

\includegraphics[scale=0.5]{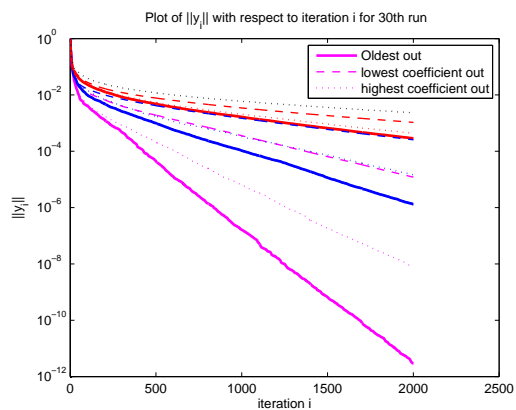}\caption{Plot of an experiment for which \eqref{eq:von-Neumann-ineq} is feasible.
The black dotted plot is for $A_{2}$ (the von Neumann algorithm).
The red plots are for the three versions of aggregation strategies
in Remark \ref{rem:agg_strat_2} for $A_{5}$. The blue and magenta
plots are for $A_{10}$ and $A_{15}$ respectively. }
\end{figure}

\section{Conclusion}

We introduced an improvement of the distance reduction step in the
generalized von Neumann algorithm in Algorithm \ref{alg:enhanced-vN}
by projecting onto the convex hull of a set of points $C_{i}$ using
a primal active set QP algorithm. The size of $C_{i}$, $|C_{i}|$,
can be chosen to be as large as possible, as long as each iteration
is manageable. If $|C_{i}|$ is increased, the cost of each iteration
increases, but we expect better iterates when \eqref{eq:von-Neumann-ineq-1}
is feasible. This is verified by our numerical experiments. When \eqref{eq:perceptron-ineq-1}
is feasible, we can find a solution to \eqref{eq:perceptron-ineq-1}
if $|C_{i}|$ is relatively large. But interestingly, if $|C_{i}|$
is small but bigger than $2$, the performance can be poorer than
von Neumann's algorithm on average.

On the theoretical side, Theorem \ref{thm:fin-conv-1} studies the
behavior of Algorithm \ref{alg:enhanced-vN} when $0\in\intr(S)$
and the rest of Sections \ref{sec:basic-analysis} and \ref{sec:More-analysis}
study the behavior of Algorithm \ref{alg:enhanced-vN} when $0\in\partial S$
and $A\in\mathbb{R}^{2\times n}$. A natural follow up question that
better models how Algorithm \ref{alg:enhanced-vN} can be used in
practice is to study what happens when $|C_{i}|$ is of moderate size.
It appears hard to prove such results because there is no easy formula
for the projection onto $\conv(C_{i})$ when $|C_{i}|>2$. Remark
\ref{rem:difficulties} also shows the difficulties for extending
our results to the case when $0\in\partial S$ and $A\in\mathbb{R}^{m\times n}$
for $m>2$. Nevertheless, the results here can give an idea of what
can be expected to be true in higher dimensions.
\begin{acknowledgement*}
We thank Marina Epelman for organizing Freundfest honoring Robert
M. Freund's 60th birthday, where Javier Pe\~{n}a talked on how Rob
Freund's contributions in the perceptron and von Neumann algorithms
influenced his recent work. We also thank Javier Pe\~{n}a for further
conversations.
\end{acknowledgement*}
\bibliographystyle{amsalpha}
\bibliography{../refs}

\end{document}

%% file: rot.pstex_t
\begin{picture}(0,0)%
\includegraphics{rot}%
\end{picture}%
\setlength{\unitlength}{4144sp}%
\begingroup\makeatletter\ifx\SetFigFont\undefined%
\gdef\SetFigFont#1#2#3#4#5{%
  \reset@font\fontsize{#1}{#2pt}%
  \fontfamily{#3}\fontseries{#4}\fontshape{#5}%
  \selectfont}%
\fi\endgroup%
\begin{picture}(12059,5039)(3714,-488)
\put(6031,1874){\makebox(0,0)[lb]{\smash{{\SetFigFont{20}{24.0}{\rmdefault}{\mddefault}{\updefault}{\color[rgb]{0,0,0}0}%
}}}}
\put(14221,3989){\makebox(0,0)[lb]{\smash{{\SetFigFont{34}{40.8}{\familydefault}{\mddefault}{\updefault}{\color[rgb]{0,0,0}$\partial$}%
}}}}
\end{picture}%

%% file: wedge.pstex_t
\begin{picture}(0,0)%
\includegraphics{wedge}%
\end{picture}%
\setlength{\unitlength}{4144sp}%
\begingroup\makeatletter\ifx\SetFigFont\undefined%
\gdef\SetFigFont#1#2#3#4#5{%
  \reset@font\fontsize{#1}{#2pt}%
  \fontfamily{#3}\fontseries{#4}\fontshape{#5}%
  \selectfont}%
\fi\endgroup%
\begin{picture}(2699,3226)(13479,7)
\put(16066,2549){\makebox(0,0)[lb]{\smash{{\SetFigFont{34}{40.8}{\familydefault}{\mddefault}{\updefault}{\color[rgb]{0,0,0}$\partial f$}%
}}}}
\put(14716,2774){\makebox(0,0)[lb]{\smash{{\SetFigFont{34}{40.8}{\familydefault}{\mddefault}{\updefault}{\color[rgb]{0,0,0}$h_{U}$}%
}}}}
\put(15886,1964){\makebox(0,0)[lb]{\smash{{\SetFigFont{34}{40.8}{\familydefault}{\mddefault}{\updefault}{\color[rgb]{0,0,0}$h_{L}$}%
}}}}
\end{picture}%

%% file: vN_acc.bbl
\providecommand{\bysame}{\leavevmode\hbox to3em{\hrulefill}\thinspace}
\providecommand{\MR}{\relax\ifhmode\unskip\space\fi MR }
\providecommand{\MRhref}[2]{%
  \href{http://www.ams.org/mathscinet-getitem?mr=#1}{#2}
}
\providecommand{\href}[2]{#2}
\begin{thebibliography}{Dan92b}

\bibitem[BB96]{BB96_survey}
H.H. Bauschke and J.M. Borwein, \emph{On projection algorithms for solving
  convex feasibility problems}, SIAM Rev. \textbf{38} (1996), 367--426.

\bibitem[BCK06]{BausCombKruk06}
H.H. Bauschke, P.L. Combettes, and S.G. Kruk, \emph{Extrapolation algorithm for
  affine-convex feasibility problems}, Numer. Algorithms \textbf{41} (2006),
  239--274.

\bibitem[BFV09]{BelloniFreundVempala09_MOR}
A.~Belloni, R.M. Freund, and S.~Vempala, \emph{An efficient rescaled perceptron
  algorithm for conic systems}, Math. Oper. Res. \textbf{34} (2009), no.~3,
  621--641.

\bibitem[Blo62]{Block62}
H.D. Block, \emph{The perceptron: a model for brain functioning}, Rev. Mod.
  Phys. \textbf{34} (1962), 123--135.

\bibitem[BZ05]{BZ05}
J.M. Borwein and Q.J. Zhu, \emph{Techniques of variational analysis}, Springer,
  NY, 2005, CMS Books in Mathematics.

\bibitem[CC01]{CheungCucker01_MP}
D.~Cheung and F.~Cucker, \emph{A new condition number for linear programming},
  Math. Program. \textbf{91} (2001), 163--174.

\bibitem[Cla83]{Cla83}
F.H. Clarke, \emph{Optimization and nonsmooth analysis}, Wiley, Philadelphia,
  1983, Republished as a SIAM Classic in Applied Mathematics, 1990.

\bibitem[Dan92a]{Dantzig92}
G.B. Dantzig, \emph{Bracketing to speed convergence illustrated on the von
  {N}eumann algorithm for finding a feasible solution to a linear program with
  a convexity constraint}, Technical report SOL \textbf{92} (1992), no.~6.

\bibitem[Dan92b]{Dantzig92_conv}
\bysame, \emph{An $\epsilon$-precise feasible solution to a linear program with
  a convexity constraint in 1/$\epsilon^2$ iterations independent of problem
  size}, Technical Report. Stanford University (1992).

\bibitem[DV06]{DunaganVempala06_MP}
J.~Dunagan and S.~Vempala, \emph{A simple polynomial-time rescaling algorithm
  for solving linear programs}, Math. Program. \textbf{114} (2006), no.~1,
  101--114.

\bibitem[EF00]{EpelmanFreund00}
M.~Epelman and R.~M. Freund, \emph{Condition number complexity of an elementary
  algorithm for computing a reliable solution of a conic linear system}, Math.
  Program. \textbf{88} (2000), 451--485.

\bibitem[EF02]{EpelmanFreund02}
\bysame, \emph{A new condition measure, preconditioners, and relations between
  different measures of conditioning for conic linear systems}, SIAM J. Optim.
  \textbf{12} (2002), no.~3, 627--655.

\bibitem[ER11]{EsRa11}
R.~Escalante and M.~Raydan, \emph{Alternating projection methods}, SIAM, 2011.

\bibitem[Fuk82]{Fukushima82}
M.~Fukushima, \emph{A finitely convergent algorithm for convex inequalities},
  IEEE Trans. Automat. Control \textbf{27} (1982), no.~5, 1126--1127.

\bibitem[FV99]{FreundVera99}
R.M. Freund and J.~Vera, \emph{Condition-based complexity of convex
  optimization in conic linear form via the ellipsoid algorithm}, SIAM J.
  Optim. \textbf{10} (1999), 155--176.

\bibitem[GI83]{Goldfarb_Idnani}
D.~Goldfarb and A.~Idnani, \emph{A numerically stable dual method for solving
  strictly convex quadratic programs}, Math. Programming \textbf{27} (1983),
  1--33.

\bibitem[GP98]{G-P98}
U.M. Garc{\'i}a-Palomares, \emph{A superlinearly convergent projection
  algorithm for solving the convex inequality problem}, Oper. Res. Lett.
  \textbf{22} (1998), 97--103.

\bibitem[GP01]{G-P01}
\bysame, \emph{Superlinear rate of convergence and optimal acceleration schemes
  in the solution of convex inequality problems}, Inherently Parallel
  Algorithms in Feasibility and Optimization and their Applications
  (D.~Butnariu, Y.~Censor, and S.~Reich, eds.), Elsevier, 2001, pp.~297--305.

\bibitem[Mor06]{Mor06}
B.S. Mordukhovich, \emph{Variational analysis and generalized differentiation
  {I} and {II}}, Springer, Berlin, 2006, Grundlehren der mathematischen
  Wissenschaften, Vols 330 and 331.

\bibitem[Nes05]{Nesterov05_SIOPT}
Y.~Nesterov, \emph{Excessive gap technique in nonsmooth convex minimization},
  SIAM J. Optim. \textbf{16} (2005), 235--249.

\bibitem[Nov62]{Novikoff62}
A.B.J. Novikoff, \emph{On convergence proofs on perceptrons}, Proceedings of
  the Symposium on the Mathematical Theory of Automata, vol. XII, 1962,
  pp.~615--622.

\bibitem[NW06]{NW06}
J.~Nocedal and S.J. Wright, \emph{Numerical optimization}, 2 ed., Springer,
  2006.

\bibitem[Pan13]{SHDQP}
C.H.J. Pang, \emph{{SHDQP}: An algorithm for convex set intersection problems
  based on supporting hyperplanes and dual quadratic programming}, ArXiv
  e-prints (2013).

\bibitem[Pan14a]{improved_SIP}
\bysame, \emph{Improved analysis of algorithms based on supporting halfspaces
  and quadratic programming for the convex intersection and feasibility
  problems}, (preprint) (2014).

\bibitem[Pan14b]{cut_Pang12}
\bysame, \emph{Set intersection problems: Supporting hyperplanes and quadratic
  programming}, Math. Programming (Online first) (2014).

\bibitem[Pie84]{Pierra84}
G.~Pierra, \emph{Decomposition through formalization in a product space}, Math.
  Programming \textbf{28} (1984), 96--115.

\bibitem[PR00]{PenaRenegar00}
J.~Pe{\~{n}}a and J.~Renegar, \emph{Computing approximate solutions for convex
  conic systems of con- straints}, Math. Program. \textbf{87} (2000), 351--383.

\bibitem[Ren95a]{Renegar_SIOPT95}
J.~Renegar, \emph{Incorporating condition measures into the complexity theory
  of linear programming}, SIAM J. Optim. \textbf{5} (1995), 506--524.

\bibitem[Ren95b]{Renegar_MP95}
\bysame, \emph{Linear programming, complexity theory and elementary functional
  analysis}, Math. Program. \textbf{70} (1995), 279--351.

\bibitem[Roc70]{Rockafellar70}
R.T. Rockafellar, \emph{Convex analysis}, Princeton, 1970.

\bibitem[Roc79]{Rockafellar79_epiL}
\bysame, \emph{Directionally {L}ipschitzian functions and subdifferential
  calculus}, Proceedings of the London Math. Soc. \textbf{3} (1979), 145--154.

\bibitem[Ros58]{Rosenblatt58}
F.~Rosenblatt, \emph{The perceptron: A probabilistic model for information
  storage and organization in the brain}, Psych. Rev. \textbf{65} (1958),
  386--408.

\bibitem[Ruj93]{Rujan93}
P.~Ruj\'{a}n, \emph{A fast method for calculating the perceptron with maximal
  stability}, J. Physics I France \textbf{3} (1993), 277--290.

\bibitem[RW98]{RW98}
R.T. Rockafellar and R.J.-B. Wets, \emph{Variational analysis}, Grundlehren der
  mathematischen Wissenschaften, vol. 317, Springer, Berlin, 1998.

\bibitem[Sch06]{Scheinberg06}
K.~Scheinberg, \emph{An efficient implementation of an active set method for
  {SVMs}}, J. Machine Learning Research \textbf{7} (2006), 2237--2257.

\bibitem[SP12]{SoheiliPena12}
N.~Soheili and J.~Pe{\~{n}}a, \emph{A smooth perceptron algorithm}, SIAM J.
  Optim. \textbf{22} (2012), no.~2, 728--737.

\bibitem[SP13]{SoheiliPena13}
\bysame, \emph{A primal-dual smooth perceptron-von {N}eumann algorithm},
  Discrete Geometry and Optimization (K.~Bezdek et~al., ed.), vol.~69, 2013,
  pp.~303--320.

\end{thebibliography}
